%% file: main.tex
\documentclass[11pt]{article}
\usepackage[margin=1.05in]{geometry}

\usepackage[english]{babel}
\usepackage{mathtools,amssymb,amsthm}
\usepackage{booktabs}
\usepackage{placeins}
\usepackage{tikz}
\usepackage[authoryear,square]{natbib}
\setcitestyle{aysep={},notesep={, },semicolon}
\definecolor{linkblue}{RGB}{0,2,134}
\usepackage[
  colorlinks=true,
  linkcolor=linkblue,
  citecolor=linkblue,
  urlcolor=linkblue,
  pdftitle={A Complex Structure on S2 times S4},
  pdfauthor={Shengtao Guo; Ethan X. Fang; Junwei Lu},
  pdfsubject={Complex threefolds and six-manifold topology}
]{hyperref}

\title{\normalfont\LARGE A Complex Structure on $S^2\times S^4$}
\author{%
  Shengtao Guo \qquad
  Ethan X. Fang \qquad
  Junwei Lu
}
\date{}

\newtheorem{theorem}{Theorem}[section]
\newtheorem{proposition}[theorem]{Proposition}
\newtheorem{lemma}[theorem]{Lemma}
\theoremstyle{definition}

\theoremstyle{remark}

\numberwithin{equation}{section}

\newcommand{\PP}{\mathbb P}
\newcommand{\CC}{\mathbb C}
\newcommand{\RR}{\mathbb R}
\newcommand{\ZZ}{\mathbb Z}
\newcommand{\cO}{\mathcal O}
\newcommand{\dP}{\operatorname{dP}_6}

\begin{document}

\maketitle

\begin{abstract}
We show that $S^2\times S^4$ admits a complex structure. Starting from a modular family of complex
two-tori associated with the $(3,4,\infty)$ triangle group and the
compactification constructed in \citep{Alpoge2026}, we replace the period
lattice by its unique monodromy-invariant index-two superlattice and
compactify the resulting family.  In the integral Mayer--Vietoris
calculation, a primitive local class whose double is the class of a cusp
component and the unique nonzero torsion class contributed by the
multiplicity-four fibre restrict to the same class of order two on the
common boundary.  We then prove that the resulting compact complex
threefold is diffeomorphic to $S^2\times S^4$.
The proof is discovered by the Odin Automatic AI Research Agent.
\end{abstract}

\section{Introduction}\label{sec:introduction}

Calabi \citep{Calabi1958} observed that $V^2\times S^4$ is almost complex for every
closed orientable surface $V^2$, using an embedding in $\RR^7$, and
asked whether it admits a complex structure.  He contrasted this question
with the complex structures already known on certain products of
odd-dimensional spheres, including the Calabi--Eckmann examples
\citep{CalabiEckmann1953}.  The simplest instance of Calabi's question is
$V^2=S^2$.

The known almost complex structures on $S^2\times S^4$ do not by themselves
resolve the integrability problem.  As shown in
\citep[Example~5.4(2)]{GranjaMilivojevic2022}, its almost complex structures
occur in infinitely many homotopy classes distinguished by their even first
Chern classes.  The Cayley structure on the product has a nonzero
Nijenhuis tensor \citep[Section~6.2]{DaurtsevaSmolentsev2018}; no
almost complex structure orthogonal to the product of round metrics is
integrable \citep[Corollary~B]{EuhSekigawa2013}.  A cohomogeneity-one
quasi-K\"ahler structure on the product was subsequently constructed in
\citep[Theorem, p.~766]{Daurtseva2020}.
Other structured examples include an exotic almost complex circle action
\citep[Corollary~1.2]{KonstantisLindsay2025} and a torus-equivariant GKM
realization \citep[Section~8.2]{JangEtAl2026}.  Both concern
almost-complex geometry rather than integrability.  Our main result gives an
affirmative answer to Calabi's question when $V^2=S^2$.

\begin{theorem}\label{thm:main}
The smooth manifold $S^2\times S^4$ admits a complex structure.
\end{theorem}

Theorem~\ref{thm:main} is closely related to the Hopf problem for $S^6$.
\citet{BorelSerre1953} proved that, among spheres of positive dimension, only
$S^2$ and $S^6$ admit almost complex structures.  The standard octonionic almost complex structure on
$S^6$ is not integrable, and whether $S^6$ admits any integrable complex
structure is the classical Hopf problem
\citep{Hopf1948,Yau1992,AgricolaEtAl2018}.  Prior work ruled out a complex
structure orthogonal to the round metric \citep{LeBrun1987} and showed that a
hypothetical complex structure on $S^6$ could not be almost homogeneous
\citep{HuckleberryEtAl2000}.  In a recent manuscript,
\citet{Alpoge2026} constructs a compact complex threefold from the
$(3,4,\infty)$ modular family of complex two-tori.  An equivariant period
map defines the family over the thrice-punctured sphere; logarithmic
transforms complete it at the two points of finite monodromy, while a
Mumford toric degeneration completes it at the cusp.  The three local models
are then glued to the smooth torus family.  A calculation of the fundamental
group and integral homology shows that the resulting threefold is a homotopy
six-sphere, and the classification of smooth homotopy six-spheres identifies
it diffeomorphically with $S^6$.

We retain this analytic framework but change the integral lattice to obtain
the topology of $S^2\times S^4$.  The
difference between $S^6$ and $S^2\times S^4$ is already visible in cohomology: the compactification formed
with $\Lambda$ has vanishing second cohomology, whereas
$H^2(S^2\times S^4;\mathbb Z)\cong\mathbb Z$.  Using the same lattice
therefore cannot yield $S^2\times S^4$.  To retain the analytic period functions
while introducing the required divisor class, we seek the smallest
monodromy-compatible enlargement that splits the cusp fibre.  This has index
two, and monodromy invariance singles out a unique such superlattice
$\Lambda^+$.  Passing from $\Lambda$ to $\Lambda^+$ induces an isogeny
of degree two on each smooth fibre, but the decisive topological change
occurs at the cusp.  For
$\Lambda$, the cusp translations act transitively on the vertices of the
periodic triangulation, so the central fibre is irreducible.  For
$\Lambda^+$, their image has index two and hence has two vertex orbits; the
central fibre therefore splits into two components
\citep[Remark~4.9]{Alpoge2026}.  The two components supply the additional
divisor class in second cohomology.  We construct the resulting
compactification and verify
that the multiplicity-three and multiplicity-four fibres still have smooth
reduced supports after the change of lattice.

Because the cusp translations generate a sublattice of index two in the
toric fibre lattice $K$, the standard toric construction must retain both
lattices.  We establish the required form of Mumford's construction
from periodic polyhedral data
\citep[Sections~2--3 and~6, especially Corollary~6.6]{Mumford1972}.
We use the polar-coordinate model of
\citet[Proposition~3.1]{NakayamaOgus2010}; their relative rounding theorems
describe the local topology
\citep[Theorems~3.7 and~5.1]{NakayamaOgus2010}.  Usui's monoid action and
shrinking map yield the required deformation retraction
\citep[Theorem~5.2]{Usui2001}.  At the two multiple fibres, the Cartan--Leray
spectral sequence computes the cohomology, while covering degrees and
explicit mapping-torus models determine the images of the restriction maps
in integral cohomology.

The lattice change also explains the integral cohomology.  On a neighbourhood
of the cusp, the class of either component is twice a primitive class
$\kappa$.  The restriction of $\kappa$ to the common boundary is a
nonzero class of order two, so $\kappa$ does not extend over the complement
by itself.  The multiplicity-four fibre contributes a class $\alpha$ of
order two with the same boundary restriction.  Hence
$(\kappa,\alpha)$ lies in the kernel of the Mayer--Vietoris difference map
and defines a primitive global class $x$.  The full topological calculation
then shows
that $Y$ is simply connected, with $H^2(Y;\mathbb Z)=\mathbb Zx$,
$H^3(Y;\mathbb Z)=0$, and torsion-free homology; it also shows that the two
cusp components represent $2x$ and $-2x$.  Thus passing to
$\Lambda^+$ creates exactly the free class in second cohomology required by
$S^2\times S^4$, without introducing torsion.

We identify the underlying smooth manifold using Wall's classification \citep{Wall1966}.
\citet[Theorem~5]{Wall1966} proved
that an oriented simply connected spin six-manifold with torsion-free
homology is determined up to oriented diffeomorphism by $H^2$, $H^3$, the
integral cubic form on $H^2$, and the first Pontryagin-class functional.  We further prove that $Y$ is spin and that its
cubic form and Pontryagin functional vanish.  Its Wall invariants therefore
agree with those of $S^2\times S^4$, and Wall's theorem identifies the two
smooth manifolds.

\vspace{3pt}

\noindent\textbf{The role of AI in this proof.} We use
Odin Automatic AI Research Agent to find the proof in this paper.

\vspace{3pt}

\noindent\textbf{Paper organization.}
Section~\ref{sec:key-inputs} assembles the proof of
Theorem~\ref{thm:main} from the results established in the next three
sections.  The analytic construction is carried out in
Section~\ref{sec:analytic-construction}, the fundamental group and integral
homology are determined
in Section~\ref{sec:topology}, and Section~\ref{sec:characteristic-classes}
computes the characteristic classes, completing the input to Wall's theorem.
Appendix~\ref{app:period-torsor} proves the existence and normalization of
the period functions used in the analytic construction.

\section{Proof sketch of the main theorem}\label{sec:key-inputs}

Let $B=\PP^1$, with two orbifold points $p_1,p_2$ of orders three and
four and a cusp $p_0$.  Choose the affine coordinate $t$ on
$B\setminus\{p_0\}$ such that $t(p_1)=0$ and $t(p_2)=1$, and put
$B^\circ=B\setminus\{p_0,p_1,p_2\}$.  Theorem~3.4 of
\citet{Alpoge2026} constructs functions $\tau,\mu,\beta$ that define a
family of complex two-tori over $B^\circ$.

The homology marking in \citet[Section~2.1]{Alpoge2026} is
$\Lambda=\ZZ\langle\widehat\gamma,\widehat u,\widehat w,
\widehat\delta\rangle$.  Put
$d=\widehat\delta/2\in\Lambda\otimes\mathbb Q$ and set
$\Lambda^+=\Lambda+\ZZ d
=\ZZ\langle\widehat\gamma,\widehat u,\widehat w,d\rangle$.
This is the unique monodromy-invariant index-two extension of $\Lambda$.
The affine actions of orders three and four remain free, and the resulting
multiple fibres have smooth reduced supports.  Put
$K=\ZZ\langle\widehat w,d\rangle$, let $t_c=t^{-1}$ be a local
coordinate at the cusp, and choose its logarithmic lift
$e^{2\pi i s}=t_c$.  The degenerating period block is
\[
 Z^+(s)=sB_0^++C^+(t_c),
 \qquad
 B_0^+=\begin{pmatrix}0&1\\-2&0\end{pmatrix},
\]
where $C^+(t_c)$ is holomorphic at the origin.  With respect to the ordered
bases $(\overline{\widehat\gamma},\overline{\widehat u})$ of
$\Lambda^+/K$ and $(\widehat w,d)$ of $K$, the image of
$B_0^+$ is $\ZZ\widehat w\oplus\ZZ(2d)$.  It has index two in
$K$, so the periodic triangulation has two vertex orbits.  Although
$B_0^+\mathbb Z^2$ has finite index in $K$, the corresponding toric
quotient is smooth and proper, and its central fibre has two components.

The construction and proof of the following compactification result are given
in Section~\ref{sec:analytic-construction}.

\begin{proposition}[Analytic compactification]\label{prop:analytic-model}
There is a compact connected smooth complex threefold $Y$ and a proper
holomorphic map $f:Y\to B$.  Over $B^\circ$, its fibres are the complex
two-tori determined by the period functions and the lattice $\Lambda^+$.  The
scheme-theoretic fibres over $p_1,p_2$ are $3S_1$ and $4S_2$, where
$S_1,S_2$ are smooth and reduced.  The cusp fibre is a reduced
normal-crossings divisor, but not a simple normal-crossings divisor.  It has
exactly two irreducible components $D_0,D_1$; each is nonnormal, and its two
local branches meet normally along a smooth rational curve.  The
normalization of either component is the toric del Pezzo surface
$\dP=\operatorname{Bl}_3\PP^2$.
\end{proposition}

For a Cartier divisor $D$ on $Y$, write
$[D]=c_1(\mathcal O_Y(D))\in H^2(Y;\mathbb Z)$, and use the same notation
for its restriction to either piece of the Mayer--Vietoris decomposition.

For the topological calculation, let $A$ be a closed cusp neighbourhood
and let $Y_{\mathrm{fin}}$ be the inverse image of a complementary closed
disc containing $p_1$ and $p_2$.  Then
$Y=A\cup_M Y_{\mathrm{fin}}$, where $M$ is the common boundary.  On
$A$ we construct a primitive class
$\kappa\in H^2(A;\ZZ)$ with $[D_0]=2\kappa$.  Its restriction to $M$ is
a nonzero class $\vartheta\in H^2(M;\ZZ)$ of order two.  Independently, the
multiplicity-four fibre contributes a nonzero class
$\alpha\in H^2(Y_{\mathrm{fin}};\ZZ)$ of order two, and the integral
restriction calculation gives
$\kappa|_M=\vartheta=\alpha|_M$.
This equality determines the torsion part of the Mayer--Vietoris difference
map.  The free part must also be computed; Section~\ref{sec:topology} carries
out the full boundary calculation and proves the following result.

\begin{proposition}[The Mayer--Vietoris map in cohomological degree two]
\label{prop:degree-two-mv}
For the difference map
\[
 H^2(A;\ZZ)\oplus H^2(Y_{\mathrm{fin}};\ZZ)
 \longrightarrow H^2(M;\ZZ),
 \qquad (\xi,\eta)\longmapsto \xi|_M-\eta|_M,
\]
the kernel is infinite cyclic, generated by
\begin{equation}\label{eq:x-generator}
 x=(\kappa,\alpha).
\end{equation}
Viewed as pairs of restrictions to $A$ and $Y_{\mathrm{fin}}$,
$[D_0]=2x$ and $[D_1]=-2x$.
\end{proposition}

The same section combines the boundary calculation with van Kampen and the
Euler characteristic of the cusp fibre.  This gives simple connectivity and,
including the integral torsion calculation, determines the remaining
homology groups, proving the following proposition.

\begin{proposition}[Fundamental group and integral homology]\label{prop:integral-topology}
The threefold $Y$ is simply connected and
\[
 H_i(Y;\ZZ)=
 \begin{cases}
  \ZZ,&i=0,2,4,6,\\
  0,&i=1,3,5.
 \end{cases}
\]
For a primitive generator $x\in H^2(Y;\ZZ)$, after choosing its sign,
$[D_0]=2x$ and $[D_1]=-2x$.
\end{proposition}

It remains to compute the invariants used in Wall's theorem.  Choose the sign
of $x$ so that $[D_0]=2x$.  The normalization calculation gives
$D_0^3=0$ and $D_0\cdot C_{\mathrm{dbl}}=-2$.  Thus $x^3=0$ and
$x(C_{\mathrm{dbl}})=-1$, so the double curve generates
$H_2(Y;\mathbb Z)$.  Since
$\deg N_{C_{\mathrm{dbl}}/Y}=-2$, adjunction gives $c_1(Y)=0$.
Finally, the normalization sequence gives $\chi(\mathcal O_{D_0})=0$, and
the Hirzebruch--Riemann--Roch theorem gives $D_0\cdot c_2(Y)=0$.  Since
$[D_0]=2x$ and $H^4(Y;\mathbb Z)\cong\mathbb Z$, this forces
$c_2(Y)=0$.  Section~\ref{sec:characteristic-classes} supplies these
normalization, adjunction, and Riemann--Roch calculations and proves the
following proposition.

\begin{proposition}[Characteristic classes]
\label{prop:characteristic-classes}
For the primitive generator $x$ fixed above,
\[
 x^3=0,\qquad c_1(Y)=0,\qquad c_2(Y)=0,
 \qquad p_1(Y)=0.
\]
\end{proposition}

\begin{proof}[Proof of Theorem~\ref{thm:main}]
Proposition~\ref{prop:analytic-model} supplies the compact complex
threefold.  By Proposition~\ref{prop:integral-topology}, it is simply
connected with torsion-free homology, $H^2(Y;\ZZ)=\ZZ x$, and
$H^3(Y;\ZZ)=0$.  Proposition~\ref{prop:characteristic-classes} gives
$x^3=0$, $w_2(Y)=0$, and $p_1(Y)=0$.  Thus the cubic and Pontryagin
functionals vanish.  Wall's classification, in the form used here, says that
an oriented simply connected spin six-manifold with torsion-free homology is
determined by $H^2$, $H^3$, the integral cubic form on $H^2$, and the
$p_1$-functional, subject to
$\mu(a,a,b)\equiv\mu(a,b,b)\pmod2$ and
$p_1(a)\equiv4\mu(a,a,a)\pmod{24}$
\citep[Theorem~5]{Wall1966}.  These congruences are automatic here, and the
displayed data are the oriented Wall invariants of $S^2\times S^4$.
The classification therefore gives an
orientation-preserving diffeomorphism
$Y\cong_{\mathrm{diff}}S^2\times S^4$, proving the theorem.
Wall's corrigendum \citep{Wall1967} concerns Theorem~14 and does not affect
this application of Theorem~5.\qedhere
\end{proof}

\input{sections/analytic}
\input{sections/topology}

\appendix
\input{sections/period-torsor-appendix}

\bibliographystyle{plainnat}
\bibliography{reference}

\end{document}

%% file: sections/analytic.tex
\section{The analytic construction}\label{sec:analytic-construction}

We prove Proposition~\ref{prop:analytic-model} by compactifying the torus
family determined by $\Lambda^+$ over the three marked points of $B$.
The rational monodromy operators and the holomorphic period functions from
\citep[Sections~2--3]{Alpoge2026} are retained; what changes is the integral
marking.  The proof is therefore organized around the three places where
that marking enters the construction.  We first check that the finite
monodromies preserve $\Lambda^+$ and admit free affine lifts.  We then
rewrite the period matrix in the new marking.  Finally, we compactify its
logarithmic term at the cusp, where the deck-translation lattice has index
two in the toric lattice.  Once these compatibility statements are in
place, it remains to identify the local models over the punctured coordinate
discs and glue them.  The period-function lemma used below is proved in
Appendix~\ref{app:period-torsor}.

\subsection{The enlarged lattice and the finite monodromies}

Let $B=\PP^1$, with $p_1=0$, $p_2=1$, and $p_0=\infty$, and let
$t$ be the affine coordinate on $B\setminus\{p_0\}$.  Following
\citep[Section~2.1]{Alpoge2026}, let
$\Lambda=\ZZ\langle\widehat\gamma,\widehat u,\widehat w,
\widehat\delta\rangle$ and
$\mathcal B=(\widehat\gamma,\widehat u,\widehat w,\widehat\delta)$.
Let $A_1,A_2\in\operatorname{Aut}(\Lambda)$ be the monodromy operators
about $p_1,p_2$.  With the column-vector convention, their matrices in
the ordered basis $\mathcal B$ are
\begin{equation}\label{eq:an-original-monodromy-matrices}
 [A_1]_{\mathcal B}=
 \begin{pmatrix}
 1&0&0&0\\
 6&0&1&0\\
 -6&-1&-1&0\\
 -2&1&0&1
 \end{pmatrix},
 \qquad
 [A_2]_{\mathcal B}=
 \begin{pmatrix}
 1&0&0&0\\
 0&0&-1&0\\
 -6&1&0&0\\
 3&0&1&1
 \end{pmatrix}.
\end{equation}
These monodromy operators and the translation vectors used below are taken
from \citep[Lemmas~2.4 and~2.6]{Alpoge2026}.  Our task is to find an
overlattice that they preserve and that has the required cusp behaviour.

To make that requirement precise, let $K$ denote the cocharacter lattice
of the torus used in the cusp compactification and let $L\subset K$ be the
lattice of deck translations.  The irreducible components of the central
fibre are indexed by $K/L$.  For $\Lambda$, the lattice $L$ is
saturated in $K$, and the central fibre has one component.  To obtain the
divisor data needed for nonzero second cohomology, we require two components,
hence $[K:L]=2$.  We therefore seek the smallest monodromy-invariant
enlargement of $\Lambda$ for which the logarithmic period matrix has this
index.  Every index-two overlattice of
$\Lambda$ lies in $\tfrac12\Lambda$ and has the form
$\Lambda+\mathbb Z(v/2)$ for a nonzero class
$\bar v\in\Lambda/2\Lambda$.  Reducing
\eqref{eq:an-original-monodromy-matrices} modulo two gives
\[
 \begin{aligned}
 \ker(\bar A_1-I)
  &=\mathbb F_2\langle\overline{\widehat\gamma},
       \overline{\widehat\delta}\rangle,\\
 \ker(\bar A_2-I)
  &=\mathbb F_2\langle
       \overline{\widehat\gamma+\widehat u+\widehat w},
       \overline{\widehat\delta}\rangle,
 \end{aligned}
 \qquad
 \ker(\bar A_1-I)\cap\ker(\bar A_2-I)
 =\mathbb F_2\,\overline{\widehat\delta}
 \subset\Lambda/2\Lambda.
\]
Thus the only possible new half-period is represented by
$\widehat\delta/2$.  Put
$d=\widehat\delta/2\in\Lambda\otimes\mathbb Q$ and
$\Lambda^+=\Lambda+\ZZ d
=\ZZ\langle\widehat\gamma,\widehat u,\widehat w,d\rangle$.
Then $[\Lambda^+:\Lambda]=2$, and $\Lambda^+$ is the unique index-two
overlattice of $\Lambda$ preserved by both monodromy operators.

Write $\mathcal B^+=(\widehat\gamma,\widehat u,\widehat w,d)$ and
$D=\operatorname{diag}(1,1,1,\tfrac12)$.
The columns of $D$ are the vectors of $\mathcal B^+$, expressed in
the basis $\mathcal B$.  The change-of-basis formula therefore gives
\begin{equation}\label{eq:an-monodromy-matrices}
 \begin{aligned}
 [A_1]_{\mathcal B^+}
 &
 =\begin{pmatrix}
  1&0&0&0\\
  6&0&1&0\\
  -6&-1&-1&0\\
  -4&2&0&1
 \end{pmatrix},
 [A_2]_{\mathcal B^+}
 &
 =\begin{pmatrix}
  1&0&0&0\\
  0&0&-1&0\\
  -6&1&0&0\\
  6&0&2&1
 \end{pmatrix}.
 \end{aligned}
\end{equation}
In particular, these matrices are integral, which verifies directly that
$A_1$ and $A_2$ preserve $\Lambda^+$.  The operators on
$\Lambda\otimes\mathbb Q=\Lambda^+\otimes\mathbb Q$ have not changed;
only their coordinate matrices have changed.  All subsequent coordinate
formulas for the monodromy use the basis $\mathcal B^+$.
The affine translation vectors are
$v_1=\widehat\gamma+2\widehat u-4\widehat w$ and
$v_2=-\widehat\gamma-3\widehat u+3\widehat w$.
Let $m_1=3$, $m_2=4$, and let
$\gamma\in\operatorname{Hom}(\Lambda^+,\ZZ)$ be dual to
$\widehat\gamma$.  The local model over $p_j$ will be obtained from the
affine transformation $x\longmapsto A_jx+v_j/m_j$ on
$(\Lambda^+\otimes\mathbb R)/\Lambda^+$.
These data give a cyclic action of order $m_j$ provided
$A_j^{m_j}=I$ and $A_jv_j=v_j$; its $m_j$-th power is then translation
by the period $v_j$.  Freeness can be checked by finding an
$A_j$-invariant integral character $\gamma$ for which
$\gcd(\gamma(v_j),m_j)=1$.  Direct multiplication gives
\begin{equation}\label{eq:an-finite-orders}
 A_1^3=A_2^4=I,
 \qquad A_jv_j=v_j,
 \qquad \gamma\circ A_j=\gamma\quad (j=1,2),
 \qquad \gamma(v_1)=1,\quad\gamma(v_2)=-1.
\end{equation}

Indeed, if the $k$-th power, with $0<k<m_j$, had a fixed point, applying
$\gamma$ to the fixed-point equation would give
$k\gamma(v_j)/m_j\in\mathbb Z$, a contradiction.  Thus
\eqref{eq:an-finite-orders} records exactly the properties of the existing
monodromy data needed after the change of lattice.  We use this criterion
explicitly in the construction of the two multiple fibres.

\subsection{Period functions and equivariance}

The matrices $[A_j]_{\mathcal B^+}$ in
\eqref{eq:an-monodromy-matrices} prescribe the integral monodromy.  To
realize them by a holomorphic family of complex tori, we need a period matrix
$\Pi^+$ and
holomorphic changes of fibre coordinates $R_j(z)\in\operatorname{GL}_2(\CC)$
satisfying $R_j(z)\Pi^+(z)=\Pi^+(g_jz)A_j$.
The functions in the next lemma provide such a matrix.  Their cusp
normalization also separates its logarithmic term from the part that extends
holomorphically across $p_0$.

Let $\Gamma_{\mathrm{orb}}
=\langle g_1,g_2\mid g_1^3=g_2^4=1\rangle$, with
$g_0=(g_1g_2)^{-1}$, act on $\mathfrak H_z$, and let
$\pi:\mathfrak H_z\to B\setminus\{p_0\}$ be the resulting orbifold
uniformization.  If $z_j$ is fixed by
$g_j$, choose the generator orientations and linearizing coordinates so
that $s_j(g_jz)=e^{-2\pi i/m_j}s_j(z)$, where
$(m_1,m_2)=(3,4)$.
Let $U_0^\times$ be a punctured disc about $p_0$, and choose a connected
component $\widetilde U_0$ of $\pi^{-1}(U_0^\times)$ whose stabilizer is
$\langle g_0\rangle$.

\begin{lemma}[Period functions]\label{lem:an-period-functions}
There are holomorphic functions
$\tau:\mathfrak H_z\to\mathfrak H$ and
$\mu,\beta:\mathfrak H_z\to\CC$ satisfying
\begin{align}
 \tau(g_1z)&=\frac{\tau-1}{\tau},
 &\tau(g_2z)&=-\frac1\tau,\label{eq:an-tau-laws}\\
 \mu(g_1z)&=\frac{1-\mu}{\tau},
 &\mu(g_2z)&=1+(\mu/\tau),\label{eq:an-mu-laws}\\
 \beta(g_1z)&=\beta+2-\frac{6(1-\mu)^2}{\tau},
 &\beta(g_2z)&=\beta-3-\frac{6\mu^2}{\tau}.
 \label{eq:an-beta-laws}
\end{align}
Here and below, functions without a displayed argument are evaluated at
$z$.  These laws imply $\tau(g_0z)=\tau-1$,
$\mu(g_0z)=\mu$, and $\beta(g_0z)=\beta+1$.
If $t_c=t^{-1}$, then on $\widetilde U_0$ there are a
logarithmic coordinate $s$ and a function $h$, holomorphic in $t_c$
at zero, such that
\begin{equation}\label{eq:an-log-coordinate}
 e^{2\pi is}=t_c,
 \qquad s(g_0z)=s(z)-1,
 \qquad \tau=s+h(t_c),
\end{equation}
and $\mu$ and $b_c:=\beta+\tau$ are holomorphic in $t_c$ at zero.
The additive constant in $\beta$ can be chosen so that
\begin{equation}\label{eq:an-D-negative}
 \mathcal D:=\Im\beta-\frac{6(\Im\mu)^2}{\Im\tau}<0
 \quad\text{on }\mathfrak H_z.
\end{equation}
\end{lemma}

\citet[Theorem~3.4]{Alpoge2026} establishes the transformation laws and cusp
behaviour recorded in Lemma~\ref{lem:an-period-functions}.
Appendix~\ref{app:period-torsor} gives a self-contained proof in the notation
of this paper, following the argument in
\citep[Sections~3.1--3.4]{Alpoge2026}.
The three parts of the lemma serve different purposes below.  The
transformation laws give the required period equivariance, the expressions
at $p_0$ determine the toric degeneration, and the inequality
$\mathcal D<0$ guarantees that the four period vectors span a real
four-dimensional lattice.

\subsection{The smooth torus family}

In the basis $\mathcal B$, the period matrix of
\citep[Definition~3.3]{Alpoge2026} is
\[
 \Pi(z)=
 \begin{pmatrix}
  6\mu(z)&\tau(z)&1&0\\
  \beta(z)&\mu(z)&0&1
 \end{pmatrix}.
\]
Changing the marking from $\mathcal B$ to $\mathcal B^+$ multiplies the
period matrix on the right by the same change-of-basis matrix $D$.
The right-hand $2$-by-$2$ block of $\Pi D$ is
$\operatorname{diag}(1,\tfrac12)$.  We restore the standard normalization
$[Z^+\mid I_2]$ by the fixed change of fibre coordinates
$S=\operatorname{diag}(1,2)$.  Accordingly, with respect to
$\mathcal B^+$, define
\begin{equation}\label{eq:an-period-matrix}
 \Pi^+(z)=
 \begin{pmatrix}
 6\mu(z)&\tau(z)&1&0\\
 2\beta(z)&2\mu(z)&0&1
 \end{pmatrix}.
\end{equation}
Equivalently, $\Pi^+=S\Pi D$.
The identity block is useful at the cusp: coordinatewise exponentiation then
quotients exactly by the period sublattice $K$, as used in
\eqref{eq:an-cusp-exponential-map} below.
Left multiplication by $S\in\operatorname{GL}_2(\CC)$ does not change the
isomorphism class of the torus.  In the basis of $\Lambda^+$, the
sublattice $\Lambda$ consists of the vectors whose $d$-coordinate is
even; hence $\Pi^+(\Lambda)=S\Pi(\Lambda)$.  It follows that the torus
defined by $\Pi$ and $\Lambda$ is isomorphic to
$\mathbb C^2/\Pi^+(\Lambda)$.  The inclusion
$\Pi^+(\Lambda)\subset\Pi^+(\Lambda^+)$ then induces the degree-two
isogeny
\[
 \mathbb C^2/\Pi^+(\Lambda)
 \longrightarrow \mathbb C^2/\Pi^+(\Lambda^+),
 \qquad
 \ker\cong\Lambda^+/\Lambda\cong\mathbb Z/2.
\]
Thus the holomorphic functions defining the periods are unchanged, up to a
fixed fibre coordinate transformation, while the integral torus is replaced
by its quotient under an isogeny of degree two.  This distinction is what
changes the integral cusp degeneration below.  Since
$\Pi^+=[Z^+\mid I_2]$,
\begin{equation*}
 \det_{\mathbb R}\Pi^+
 =-\det\Im Z^+
 =2\Im\tau
   \left(\Im\beta-\frac{6(\Im\mu)^2}{\Im\tau}\right)
 =2\Im\tau\,\mathcal D\ne0.
\end{equation*}
Thus $\Pi^+(z)\Lambda^+$ is a real lattice in $\CC^2$.

The automorphy factors required for descent are determined by the
transformation laws in Lemma~\ref{lem:an-period-functions}.  Explicitly, put
\begin{equation*}
 R_1(z)=
 \begin{pmatrix}
 -1/\tau&0\\2(1-\mu)/\tau&1
 \end{pmatrix},
 \qquad
 R_2(z)=
 \begin{pmatrix}
 1/\tau&0\\-2\mu/\tau&1
 \end{pmatrix}.
\end{equation*}
Multiplication of the displayed matrices, using
\eqref{eq:an-tau-laws}--\eqref{eq:an-beta-laws}, gives
\begin{equation}\label{eq:an-period-equivariance}
 R_j(z)\Pi^+(z)=\Pi^+(g_jz)A_j.
\end{equation}
This identity says that changing a lift in the orbifold cover carries each
marked period vector to the period vector prescribed by the integral
monodromy.  It is therefore the descent condition for the marked torus
family.
For a word $g$ in $g_1,g_2$, compose the corresponding $R_j$'s to
obtain $R_g$.  Iterating \eqref{eq:an-period-equivariance} around the
relations $g_1^3=g_2^4=1$, and using that $\Pi^+$ has rank two over
$\mathbb C$, shows that $R_g$ is well defined and
$R_{gh}(z)=R_g(hz)R_h(z)$.  The cusp laws also give $R_{g_0}=I$.

Let $\Lambda^+$ act on $\mathfrak H_z\times\CC^2$ by
$\lambda\cdot(z,\zeta)=(z,\zeta+\Pi^+(z)\lambda)$.
Equation \eqref{eq:an-period-equivariance} shows that the action
$g\cdot(z,[\zeta])=(gz,[R_g(z)\zeta])$
normalizes this lattice action.  After removing the two elliptic orbits, the
action on $\mathfrak H_z$ is free and properly discontinuous, and the
quotient base is $B^\circ$.  The resulting quotient is therefore a proper
holomorphic torus family
$J\to B^\circ:=B\setminus\{p_0,p_1,p_2\}$.
Indeed, over every compact subset of $B^\circ$ this is a smooth bundle
with compact real four-torus fibre, which is the properness assertion used
in the final gluing.

The period matrix already displays the cusp degeneration.  From
$\tau=s+h$ and $\beta=b_c-\tau$, one obtains
\begin{equation}\label{eq:an-cusp-period-splitting}
 Z^+(s)=sB_0^++C^+(t_c),
 \qquad
 B_0^+=
 \begin{pmatrix}0&1\\-2&0\end{pmatrix},
\end{equation}
where
\begin{equation}\label{eq:an-cusp-C}
 C^+(t_c)=
 \begin{pmatrix}
 6\mu&h\\2b_c-2h&2\mu
 \end{pmatrix}
\end{equation}
is holomorphic at $t_c=0$.  For
$K=\ZZ\widehat w\oplus\ZZ d$, the image of $B_0^+$ is
\begin{equation}\label{eq:an-index-two-image}
 B_0^+\ZZ^2
 =\ZZ\widehat w\oplus\ZZ(2d)\subset K,
\end{equation}
which has index two.  More explicitly, in the ordered bases
$(\overline{\widehat\gamma},\overline{\widehat u})$ of
$\Lambda^+/K$ and $(\widehat w,d)$ of $K$,
$B_0^+(a,b)=b\widehat w-2ad$.  This also fixes the marking used in the
topological calculation.

The mechanism behind this index is that the primitive vector
$\widehat\delta\in\Lambda$ becomes divisible in $\Lambda^+$, where
$\widehat\delta=2d$.  In the original lattice, the logarithmic coefficient
is
\[
 B_0=\begin{pmatrix}0&1\\-1&0\end{pmatrix},
 \qquad
 B_0\mathbb Z^2
 =K_\Lambda:=\mathbb Z\widehat w\oplus\mathbb Z\widehat\delta,
\]
so the cusp translation lattice is saturated.  Its action on the vertices
of the periodic triangulation has one orbit.  After adjoining $d$, the same
rational map sends $\overline{\widehat\gamma}$ to $-2d$ rather than to a
primitive vector.  The deck translations therefore form the proper
sublattice
$L=B_0^+\mathbb Z^2\subset K$, and
$K/L\cong\mathbb Z/2$.
The two cosets give two irreducible components of the central fibre.  Thus
the rational monodromy and period asymptotics are unchanged, but the quotient
indexing the components changes from the trivial group to $\mathbb Z/2$.
This is why the lattice used in the $S^6$ construction cannot give the
target topology here: the corresponding compactification has
vanishing integral second cohomology \citep{Alpoge2026}, whereas the two
components for $\Lambda^+$ supply the divisor data from which the generator of
$H^2(S^2\times S^4;\mathbb Z)$ is obtained.  Section~\ref{sec:topology}
makes this implication precise.

If $A_0=(A_1A_2)^{-1}$, direct multiplication gives
\begin{equation}\label{eq:an-cusp-monodromy-interface}
 \begin{aligned}
 A_0\widehat\gamma=\widehat\gamma-2d,
 A_0\widehat u=\widehat u+\widehat w,
 A_0\widehat w=\widehat w,
 A_0d=d.
 \end{aligned}
\end{equation}
In particular, $K$ is fixed by $A_0$, and $A_0-I$ factors through a
map $\Lambda^+/K\to K$.  In the bases used in
\eqref{eq:an-index-two-image}, the matrix of this map is precisely
$B_0^+$.  Thus $B_0^+$ is not an additional choice: it is the matrix of
the homomorphism induced by $A_0-I$, whose nonprimitive image must be
retained in the toric quotient.
In Section~\ref{sec:topology} we retain $\widehat\gamma$ and abbreviate
only $\widehat u,\widehat w$ to $u,w$.

\subsection{The toric model at the cusp}

The logarithmic expression \eqref{eq:an-cusp-period-splitting} places the
cusp in the setting of Mumford's toroidal construction
\citep[Sections~2--3]{Mumford1972}.  Here $K$ is the
cocharacter lattice of the dense torus, while
$L=B_0^+\mathbb Z^2$ is the lattice of deck translations.  A
$K$-periodic unimodular triangulation gives a smooth toric ambient space,
and placing all ray generators at height one makes the central divisor
reduced.  The quotient by the deck group identifies vertices precisely along
$L$-orbits, so its irreducible components are indexed by $K/L$.  Unlike
the saturated case, we therefore must keep $K$ and $L$ distinct
throughout the construction.

Choose a basis $e_1,e_2$ of a rank-two lattice $K$.  Let $\mathcal T$
be the periodic triangulation of $K_{\mathbb R}$ with triangles
\[
 \operatorname{conv}(v,v+e_1,v+e_2),
 \qquad
 \operatorname{conv}(v+e_1,v+e_2,v+e_1+e_2),\qquad v\in K.
\]

For the lattice $\Lambda^+$, we have
$e_1=\widehat w$, $e_2=d$, and
$L=B_0^+\mathbb Z^2=\mathbb Ze_1\oplus\mathbb Z(2e_2)$.
Thus $K/L\cong\mathbb Z/2$, and its two cosets are the two vertex orbits.
Write $v_{ij}=ie_1+je_2$ and $\widehat v=(v,1)$.  In a fundamental
strip for $L$, the four rays satisfy the primitive circuit
$\widehat v_{10}+\widehat v_{02}
=\widehat v_{01}+\widehat v_{11}$.
The triangulation $\mathcal T$ chooses the diagonal
$[v_{01},v_{11}]$.  Its horizontal translates identify an opposite pair
of boundary curves within each component, while the other four
boundary curves meet the opposite component, giving the configuration of
double curves used in Proposition~\ref{prop:characteristic-classes}.
Figure~\ref{fig:cusp-combinatorics} shows the fundamental strip, the quotient
dual graph, and the toric boundary of a normalized component.  Choose one
lift of each vertex of the
quotient dual graph.  For an oriented edge, its translation label is the
element of $L$ carrying the chosen lift of its terminal vertex to the
terminal vertex of the lifted edge.

\begin{figure}[tb]
\centering
\begin{tikzpicture}[
  every node/.style={font=\scriptsize},
  panel title/.style={font=\small},
  orbit zero vertex/.style={circle,draw,fill=black,inner sep=1.45pt},
  orbit one vertex/.style={circle,draw,fill=white,inner sep=1.45pt},
  within orbit edge/.style={black,line width=0.95pt},
  between orbit edge/.style={gray!70,line width=0.55pt},
  dual vertex/.style={circle,draw,fill=white,minimum size=5.7mm,inner sep=0pt},
  dual edge/.style={->,line width=0.55pt,shorten <=1pt,shorten >=1pt},
  graph label/.style={fill=white,inner sep=0.7pt},
  identification/.style={dashed,->,line width=0.55pt,shorten <=1.5pt,shorten >=1.5pt},
  >=stealth]
\path[use as bounding box] (-2.05,-1.52) rectangle (11.15,2.05);

\begin{scope}
  \node[panel title] at (0,1.80) {fundamental $L$-strip};
  \draw[dashed,gray!70] (-0.61,-0.49) rectangle (0.61,1.39);
  \draw[between orbit edge] (-0.48,-0.35)--(-0.48,1.25) (0.48,-0.35)--(0.48,1.25);
  \draw[between orbit edge] (0.48,-0.35)--(-0.48,0.45) (0.48,0.45)--(-0.48,1.25);
  \draw[within orbit edge] (-0.48,-0.35)--(0.48,-0.35) (-0.48,0.45)--(0.48,0.45) (-0.48,1.25)--(0.48,1.25);
  \node[orbit zero vertex] at (-0.48,-0.35) {};
  \node[orbit zero vertex] at (0.48,-0.35) {};
  \node[orbit one vertex] at (-0.48,0.45) {};
  \node[orbit one vertex] at (0.48,0.45) {};
  \node[orbit zero vertex] at (-0.48,1.25) {};
  \node[orbit zero vertex] at (0.48,1.25) {};
  \node[below right=1pt] at (0.48,-0.35) {$v_{10}$};
  \node[left=2pt] at (-0.48,0.45) {$v_{01}$};
  \node[right=2pt] at (0.48,0.45) {$v_{11}$};
  \node[left=2pt] at (-0.48,1.25) {$v_{02}$};
  \draw[->,line width=0.55pt] (-0.48,-0.78)--(0.48,-0.78);
  \node[right=2pt] at (0.48,-0.78) {$e_1$};
  \draw[->,line width=0.55pt] (-1.18,-0.35)--(-1.18,0.45) node[midway,left=2pt] {$e_2$};
  \node[align=center] at (0,-1.28) {filled/open: $j\equiv0/1\pmod2$\\black/gray: within/between vertex orbits};
\end{scope}

\begin{scope}[xshift=4.55cm]
  \node[panel title] at (0,1.80) {dual cell graph};
  \node[dual vertex] (L0) at (-1.14,0.05) {$\mathsf L_0$};
  \node[dual vertex] (U0) at (0,-0.82) {$\mathsf U_0$};
  \node[dual vertex] (L1) at (1.14,0.05) {$\mathsf L_1$};
  \node[dual vertex] (U1) at (0,0.92) {$\mathsf U_1$};
  \draw[dual edge] (L0) to[bend right=17] (U0);
  \draw[dual edge] (L0) to[bend left=17] (U0);
  \draw[dual edge] (L0) -- (U1);
  \draw[dual edge] (L1) to[bend left=17] (U1);
  \draw[dual edge] (L1) to[bend right=17] (U1);
  \draw[dual edge] (L1) -- (U0);
  \node[graph label] at (-0.83,-0.65) {$d_0$};
  \node[graph label] at (-0.50,-0.31) {$\upsilon_0$};
  \node[graph label] at (-0.70,0.66) {$h_0$};
  \node[graph label] at (0.83,0.75) {$d_1$};
  \node[graph label] at (0.50,0.41) {$\upsilon_1$};
  \node[graph label] at (0.70,-0.56) {$h_1$};
  \node at (0,-1.36) {faces $P_0,P_1$};
\end{scope}

\begin{scope}[xshift=9.10cm]
  \node[panel title] at (0,1.80) {boundary of $\dP$};
  \coordinate (c0) at (1.07,0); \coordinate (c1) at (0.535,0.927);
  \coordinate (c2) at (-0.535,0.927); \coordinate (c3) at (-1.07,0);
  \coordinate (c4) at (-0.535,-0.927);
  \coordinate (c5) at (0.535,-0.927);
  \draw[within orbit edge] (c0)--(c1);
  \draw[between orbit edge] (c1)--(c2)--(c3);
  \draw[within orbit edge] (c3)--(c4);
  \draw[between orbit edge] (c4)--(c5)--(c0);
  \node[right=2pt] at (0.80,0.53) {$C_1$};
  \node[above=2pt] at (0,0.927) {$C_2$};
  \node[left=2pt] at (-0.80,0.53) {$C_3$};
  \node[left=2pt] at (-0.80,-0.53) {$C_4$};
  \node[below=2pt] at (0,-0.927) {$C_5$};
  \node[right=2pt] at (0.80,-0.53) {$C_6$};
  \draw[identification] (0.80,0.50) .. controls (0.50,-0.15) and (-0.45,-0.18) .. node[pos=0.54,below=1pt] {$\phi$} (-0.80,-0.50);
\end{scope}
\end{tikzpicture}
\caption{The cusp for the index-two lattice.  In the left panel, the middle
horizontal edge is the diagonal $[v_{01},v_{11}]$ selected by
$\mathcal T$, and the dashed
rectangle is a fundamental region for
$L=\mathbb Ze_1\oplus\mathbb Z(2e_2)$.  The middle panel is the quotient
dual graph.  In the indicated edge order
$(d_0,\upsilon_0,h_0,d_1,\upsilon_1,h_1)$, its translation labels are
$(0,-e_1,-2e_2,0,-e_1,0)$.  In the right panel, the two heavy curves
$C_1,C_4$ are identified by $\phi$ to form a double curve of the
component, while $C_2+C_3+C_5+C_6$ maps to its intersection with the other
component.}
\label{fig:cusp-combinatorics}
\end{figure}
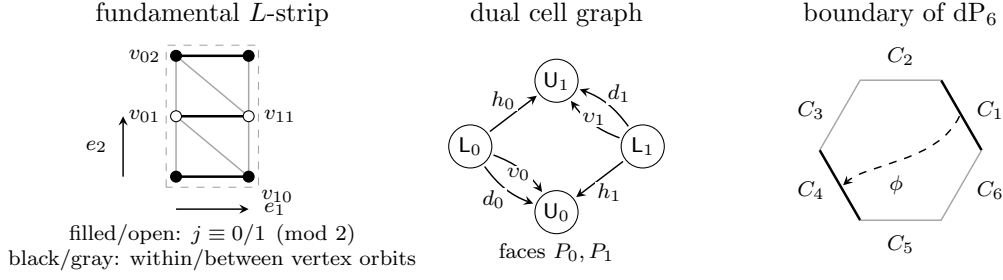
\FloatBarrier

The local fan construction is the same as in the saturated case.  What must
be checked here is that a proper finite-index sublattice still acts freely
and properly discontinuously and that the quotient remains proper over the
disc.  The following lemma records this form of the toric compactification;
compare \citep[Theorem~4.5]{Alpoge2026}.

\begin{lemma}[Toric compactification at the cusp]\label{lem:an-finite-index-cusp}
Let $\mathcal T$ be the periodic triangulation just defined, and let
$\Sigma$ be the fan in
$K_{\mathbb R}\oplus\mathbb R$ consisting of the cones over its faces
at height one.  Let $B_0:\ZZ^2\to K$ be injective with finite-index
image, and let $C(t)$ be a holomorphic $2$-by-$2$ matrix near zero.
The height character $t:=\chi^{(0,1)}$ extends to a holomorphic map
$t:X_\Sigma\to\CC$.  On the dense torus, write $x$ for the coordinate
in the two-dimensional torus with cocharacter lattice $K$, and define
\begin{equation}\label{eq:an-toric-action}
 \Psi_\lambda(x,t)
 =\bigl(e^{2\pi iC(t)\lambda}t^{B_0\lambda}x,t\bigr),
 \qquad \lambda\in\ZZ^2,
\end{equation}
where exponentials and powers are taken coordinatewise in the chosen basis
of $K$.  The maps extend to biholomorphisms.  For sufficiently
small $\epsilon>0$, their action on
$X_{\Sigma,\epsilon}:=\{|t|<\epsilon\}$ is free and properly
discontinuous, and
$X_{\Sigma,\epsilon}/\ZZ^2\to\{|t|<\epsilon\}$ is a proper map from a
smooth complex threefold.  Its central fibre is a reduced divisor with
normal crossings.
\end{lemma}

\begin{proof}
The integral shear $(y,r)\mapsto(y+rB_0\lambda,r)$
preserves $\Sigma$: at height one it translates $\mathcal T$ by the
lattice element $B_0\lambda$.  The induced toric map, composed with the
torus translation in \eqref{eq:an-toric-action}, gives the desired
extension.  The shear is the identity on the fibre cocharacter lattice and
fixes $t$, so it commutes with the torus factor.  These extensions form a
$\ZZ^2$-action because $\lambda\mapsto C(t)\lambda$ is additive.

For $t\ne0$, put $y=\log|x|/\log|t|$, coordinatewise; this is the
normalized logarithmic coordinate.  Then
\begin{equation}\label{eq:an-log-coordinate-translation}
 y\longmapsto y+B_t\lambda,
 \qquad
 B_t=B_0-\frac{2\pi\Im C(t)}{\log|t|}.
\end{equation}
Since $B_0$ is invertible over $\mathbb R$, shrinking $\epsilon$
gives a constant $a>0$ such that
\begin{equation}\label{eq:an-singular-value-bound}
 \|B_t\lambda\|\geq a\|\lambda\|
 \qquad(0<|t|<\epsilon,\ \lambda\in\ZZ^2).
\end{equation}
Hence the action is free on the dense torus.  A boundary
orbit is indexed by a bounded face $P$ of $\mathcal T$, and
$\Psi_\lambda$ carries it to the orbit indexed by
$P+B_0\lambda$.  No nonzero translation stabilizes a bounded face, so
the action is free on the boundary as well.

For proper discontinuity, let
$\sigma=\{v_0,v_1,v_2\}$.  Unimodularity identifies its chart with
$\CC^3$, with coordinates $z_0,z_1,z_2$ satisfying
\begin{equation}\label{eq:an-local-t-equation}
 t=z_0z_1z_2.
\end{equation}
If $\theta_i^\sigma$ are the barycentric coordinates of $\sigma$, then
on the dense torus
\begin{equation}\label{eq:an-barycentric-chart}
 \log|z_i|=\log|t|\,\theta_i^\sigma(y).
\end{equation}
Write
$U_\sigma(S)=\{|z_i|<S\text{ for }i=0,1,2\}
\cap X_{\Sigma,\epsilon}$.
Every point of $X_{\Sigma,\epsilon}$ lies in the closed unit polydisc of
some chart.  For a torus point, choose a triangle containing $y$.  For a
point on the orbit of a face $P$, choose a triangle containing $P$ whose
quotient cone contains the negative logarithmic coordinate normal to that
orbit.  Equation \eqref{eq:an-barycentric-chart} gives the assertion in
both cases.

If
$\Psi_\lambda U_\sigma(2)\cap U_{\sigma'}(2)\ne\varnothing$, density
gives such an intersection with $t\ne0$.  Formula
\eqref{eq:an-barycentric-chart} places the two normalized logarithmic
coordinates in
fixed bounded enlargements of $\sigma$ and $\sigma'$; their difference
is $B_t\lambda$.  The lower bound
\eqref{eq:an-singular-value-bound} therefore bounds $\lambda$ for each
fixed pair $(\sigma,\sigma')$.  The sets $U_\sigma(2)$ form an open
cover.  Given two compact sets, choose
finite subcovers of each by these sets and apply the bound to the finitely
many resulting pairs of triangles.  Only finitely many group elements can
move one compact set to meet the other, which proves proper discontinuity.

For properness over the disc, first suppose $t\ne0$ and choose
$\lambda$ so that
$B_t^{-1}y+\lambda\in[-1/2,1/2]^2$.
Uniform upper bounds on $B_t$ then put the translated logarithmic coordinate
in a fixed ball.  At $t=0$, use the finite set
$K/B_0\ZZ^2$: translating a face makes one of its vertices belong to a
fixed set of coset representatives.  Local finiteness of $\mathcal T$
then leaves only finitely many adjacent triangles.  Consequently, for each
$0<\epsilon_1<\epsilon$, every orbit over
$|t|\leq\epsilon_1$ meets
\begin{equation*}
 \mathcal K_0=
 \bigcup_{\sigma\in\mathcal F}
 \{|z_0|,|z_1|,|z_2|\leq1,\ |t|\leq\epsilon_1\},
\end{equation*}
where $\mathcal F$ is a finite family of triangles.  This is compact in
$X_{\Sigma,\epsilon}$, and its image is the full inverse image of the
closed subdisc.  The quotient map to the disc is therefore proper.

Each of the two triangle types has determinant $1$ or $-1$, so
$\Sigma$ is unimodular and $X_\Sigma$ is smooth.  The free quotient is
smooth.  The fan is countable, so $X_\Sigma$ has a countable toric atlas;
the quotient map is open, and the quotient is therefore second countable.
Since every ray generator has height one,
\eqref{eq:an-local-t-equation} gives a reduced normal crossings central
fibre.
\end{proof}

In this application the lemma's local height coordinate $t$ is
$t_c$.  Apply Lemma~\ref{lem:an-finite-index-cusp} to
\eqref{eq:an-cusp-period-splitting}--\eqref{eq:an-cusp-C}.  Denote the
quotient by $N_0\to\mathbb D_{t_c}$.
By \eqref{eq:an-index-two-image}, the vertices have two orbits, so the
central fibre has two irreducible components.  The identity point
$(1,1)$ extends across the toric model.  Indeed, in the chart for the
triangle $\{0,e_1,e_2\}$, the coordinates are
$z_0=t_c/(x_1x_2)$, $z_1=x_1$, and $z_2=x_2$.
The curve $(x_1,x_2,t_c)=(1,1,t_c)$ tends to $(0,1,1)$.  Its image
is a holomorphic section over the cusp disc, which we use as the origin in
the gluing.

For a height-one vertex $v$, let $E_v$
be the corresponding toric divisor and let $D_{[v]}$ be its image in the
quotient.  The star of $v$ has six primitive rays and is the complete fan
of $\dP$, so $E_v\cong\dP$ is compact and normal.  The map
$E_v\to D_{[v]}$ is surjective, is proper with finite fibres by proper
discontinuity, and is one-to-one over the dense orbit because no nonzero
translation fixes $v$.  It is therefore finite and birational, hence is
the normalization of $D_{[v]}$.  The two edges in the
$\pm\widehat w$-directions join vertices whose second coordinates are
congruent modulo two.  Their opposite boundary curves are identified to give
a double curve of $D_{[v]}$.  The other four edge directions join vertices
whose second coordinates are not congruent modulo two and map to the
intersection with the other component.  In a
triangle chart $t_c=z_0z_1z_2$, two vertices agree modulo two in their
second coordinates and the third does not, so the corresponding divisors
have local equations
$z_0z_1=0$ and $z_2=0$.  Thus the nonnormal double curve of
$D_{[v]}$ is ordinary and unramified, including at the triple points.

\subsection{The two multiple fibres}

The change of lattice must also be compatible with the finite monodromies.
Over the orbifold coordinate $s_j$, the linear action fixes the origin of
the torus and would produce a singular quotient.  The translation by
$v_j/m_j$ is chosen to produce a free affine action.  We must verify that this
action has order $m_j$, that its quotient has central divisor $m_jS_j$
with $S_j$ smooth, and that away from the central fibre it agrees with the
family $J$.  These are the properties needed to retain the multiplicities
three and four after passing from $\Lambda$ to $\Lambda^+$.

We apply to $\Lambda^+$ the logarithmic-transform construction described in
\citep[Definition~5.3 and Theorem~5.4]{Alpoge2026}, and verify freeness and
the identifications over the punctured discs.

On a small disc about $z_j$, the torus family defined by
\eqref{eq:an-period-matrix} extends across $s_j=0$.  Define a holomorphic
action on it by
\begin{equation}\label{eq:an-holomorphic-affine-action}
 \widetilde g_j(z,\zeta)
 =\left(g_jz,
 R_j(z)\zeta+\Pi^+(g_jz)\frac{v_j}{m_j}\right).
\end{equation}
Equation \eqref{eq:an-period-equivariance} makes this action well defined
modulo the periods.  In flat coordinates it is
$\widetilde g_j(z,x)=(g_jz,A_jx+v_j/m_j)$.
By \eqref{eq:an-finite-orders}, its $m_j$-th power is translation by
$v_j\in\Lambda^+$, hence the identity on the torus.

By \eqref{eq:an-finite-orders}, $\gamma\circ A_j^k=\gamma$.  A fixed point
of the $k$-th power can occur
only over $s_j=0$, and it would imply
$(A_j^k-I)x+kv_j/m_j\in\Lambda^+$.
Applying $\gamma$ would give
$k\gamma(v_j)/m_j\in\ZZ$.  For $j=1$, the two values are
$1/3,2/3$; for $j=2$, the three values are
$-1/4,-1/2,-3/4$.  None is integral.  Thus both cyclic actions, including
all their nonidentity elements, are free.

Let $N_j$ be the quotient by this cyclic action.  It is smooth and proper
over the disc with coordinate $t_j=s_j^{m_j}$.
Indeed, before quotienting the family is, in real flat coordinates, a
bundle of compact real four-tori; over every compact subdisc its total
space is compact, and a finite quotient preserves properness.  Since the
quotient is \'etale and
$\operatorname{div}(s_j^{m_j})=m_j\{s_j=0\}$, its central divisor is
$m_jS_j$, with $S_j$ smooth and reduced.

To glue this model to $J$, we identify their restrictions over the
punctured disc.  On a simply connected sector of the punctured $s_j$-disc,
choose a branch of the logarithm and put
$\ell(z)=\log s_j(z)/(2\pi i)$.
On the $g_j$-translates of the chosen sector, choose compatible branches
satisfying $\ell(g_jz)=\ell(z)-1/m_j$.  Translation in the torus by
$\ell(z)\Pi^+(z)v_j$
conjugates \eqref{eq:an-holomorphic-affine-action} to the linear action.
Indeed, $A_jv_j=v_j$ and
\eqref{eq:an-period-equivariance} give
$\Pi^+(g_jz)v_j/m_j+\ell(g_jz)\Pi^+(g_jz)v_j
=R_j(z)\ell(z)\Pi^+(z)v_j$.
Changing the logarithm changes the translation
$\ell(z)\Pi^+(z)v_j$ by the period $\Pi^+(z)v_j$.  It therefore descends to
the torus and identifies $N_j\setminus S_j$ with $J$ over the punctured
disc.  The zero section of $J$ extends across the cusp as the holomorphic
section constructed above; we use this zero section in all three
identifications over the punctured discs.

\subsection{Identification over the punctured cusp and global gluing}

The torus family $J$ is written in additive period coordinates, whereas
the toric model $N_0$ is written in multiplicative coordinates with
cocharacter lattice $K$.  To glue them, we must show that exponentiation
first quotients by $K$ and the cusp monodromy and that the residual deck
action is exactly the action used in the toric construction.  This identifies
the two models over a punctured cusp disc.

On the chosen component $\widetilde U_0$, consider
\begin{equation}\label{eq:an-cusp-exponential-map}
 (s,\zeta)\longmapsto
 \bigl(t_c=e^{2\pi is},\ x=e^{2\pi i\zeta}\bigr),
\end{equation}
where the second exponential is coordinatewise.  It quotients the last two
period columns of \eqref{eq:an-period-matrix}.  It also quotients the base
deck transformation because $s(g_0z)=s(z)-1$ and
$R_{g_0}=I$.  Its kernel is the subgroup $H$ generated
by $g_0$ and the period lattice
$K=\mathbb Z\widehat w\oplus\mathbb Zd$.  Equations
\eqref{eq:an-cusp-monodromy-interface} give
$A_0|_K=I$ and $(A_0-I)\Lambda^+\subset K$, so $H$ is normal in the
deck group over the cusp,
$G_c=\Lambda^+\rtimes_{A_0}\langle g_0\rangle$.  The quotient
$G_c/H$ is $\Lambda^+/K\cong\mathbb Z^2$.  By
\eqref{eq:an-cusp-period-splitting}, it acts by
$x\longmapsto e^{2\pi iC^+(t_c)\lambda}
t_c^{B_0^+\lambda}x$, where $\lambda\in\ZZ^2$,
which is the toric action \eqref{eq:an-toric-action}.  Hence
\eqref{eq:an-cusp-exponential-map} induces a biholomorphism between the
restriction of $J$ to a punctured cusp disc and the punctured restriction
of $N_0$.

We use the gluing procedure of \citep[Construction~6.1]{Alpoge2026},
replacing the cusp model associated with $\Lambda$ by the one constructed
above for $\Lambda^+$.

\begin{proof}[Proof of Proposition~\ref{prop:analytic-model}]
Choose pairwise disjoint coordinate discs $U_j$ around the three marked
points.  Glue $J$, $N_0$, $N_1$, and $N_2$ by the three
biholomorphisms over $U_j\setminus\{p_j\}$, and call the result
$Y$.  The transition maps are biholomorphic, and distinct local models do
not overlap, so they give a smooth complex-manifold atlas.

The glued space is Hausdorff.  Two points with different images in $B$
are separated by inverse images
of disjoint base neighbourhoods.  Two points over the same base point lie
in a single Hausdorff model: over each $U_j$ the glued space is exactly
$N_j$, with no second copy of the central fibre, while away from the
marked points it is $J$.  It is second countable because only finitely many
second-countable spaces were glued along open subsets.

The map $f:Y\to B$ is locally one of the proper maps constructed above.
Properness is local on the target for maps between locally compact
Hausdorff spaces, so $f$ is proper.  Since $B$ is compact, $Y$ is
compact.  The space $J$ is connected, and every local model meets it over a
punctured disc, so $Y$ is connected.  The descriptions of the three
special fibres proved above give all the remaining assertions of the
proposition.
\end{proof}

%% file: sections/topology.tex
\section{Fundamental group and integral homology}\label{sec:topology}

This section proves Propositions~\ref{prop:degree-two-mv} and
\ref{prop:integral-topology}.
We adapt the deformation-retraction and Mayer--Vietoris calculations for
$\Lambda$ in
\citep[Sections~7.1--7.6 and Appendix~A]{Alpoge2026}.  For $\Lambda^+$,
the cusp fibre has two components and its boundary cohomology contains
additional classes of order two.  Integral coefficients are therefore
essential.
Throughout this section, all matrices act on column vectors.
On the homology lattice of a smooth fibre we continue to use the basis
$\Lambda^+=\mathbb Z\langle \widehat\gamma,u,w,d\rangle$,
and on its dual the basis
$V^+=\mathbb Z\langle\gamma,u^\vee,w^\vee,\Delta\rangle$, dual to
$(\widehat\gamma,u,w,d)$.  We use the contragredient convention
$T_j(\chi)=\chi\circ A_j^{-1}$, or equivalently
$T_j=(A_j^{-1})^{\mathsf T}$, for $\chi\in V^+$.
This is the action used in the cohomological Wang and Cartan--Leray spectral
sequences below.
For $j=0$, this is the cohomological monodromy associated with the clockwise
cusp meridian used below.  The cusp monodromy
$M_0:=A_0=(A_1A_2)^{-1}$ is
\begin{equation}\label{eq:cusp-monodromy-homology}
 M_0\widehat\gamma=\widehat\gamma-2d,\qquad
 M_0u=u+w,\qquad M_0w=w,\qquad M_0d=d.
\end{equation}
Under this convention, the action on the dual lattice is
\begin{equation}\label{eq:cusp-monodromy-cohomology}
 T_0\gamma=\gamma,\qquad T_0u^\vee=u^\vee,\qquad
 T_0w^\vee=w^\vee-u^\vee,\qquad T_0\Delta=\Delta+2\gamma.
\end{equation}

Choose pairwise disjoint closed discs about the three exceptional points of
the base.  On the cusp disc write $t=t_c$ for the local coordinate from
Section~\ref{sec:analytic-construction}, so that $p_0$ is given by $t=0$.
For a sufficiently small $\eta>0$, let
$A=f^{-1}(\{|t|\leq\eta\})$ be the closed cusp neighbourhood, let
$Y_{\mathrm{fin}}$ be the inverse
image of the complementary closed disc, and put
$M=A\cap Y_{\mathrm{fin}}=\partial A=\partial Y_{\mathrm{fin}}$.

\subsection{The cusp neighbourhood and its boundary}

At the cusp let
$K=\mathbb Z\langle e_1,e_2\rangle
=\mathbb Z\langle w,d\rangle$ and
$\Gamma=\mathbb Z\langle \widehat\gamma,u\rangle$,
and put $\Gamma'=2\Gamma$ and
$G=\Gamma/\Gamma'\cong(\mathbb Z/2)^2$.  This is the deck group of the
intermediate $\Gamma'$-quotient used below.  The cocharacter-translation
map is the injection
\begin{equation}\label{eq:cusp-B-marking}
 B_0^+:\Gamma\longrightarrow K,
 \qquad B_0^+(\widehat\gamma)=-2d,\qquad B_0^+(u)=w.
\end{equation}
Its image is
$L=\mathbb Zw\oplus\mathbb Z(2d)$.  Write
$T_B=K_{\mathbb R}/L$ and $T_f=K_{\mathbb R}/K$.
The real-linear extension of $B_0^+$ is an isomorphism
$\Gamma_{\mathbb R}\to K_{\mathbb R}$ carrying $\Gamma$ onto $L$, and
hence identifies $\Gamma_{\mathbb R}/\Gamma$ with $T_B$.
The chosen homology marking identifies a nearby real four-torus with
$F=T_B\times T_f$.  The
clockwise cusp meridian is denoted by $a_0$; this convention agrees with
$g_0=(g_1g_2)^{-1}$ and with the logarithmic coordinate
$t=e^{2\pi i s}$, for which $s(g_0z)=s(z)-1$.
For a monodromy-invariant fibre one-cycle $c$, write
$\operatorname{sw}_c(a_0)$ for the suspension torus obtained by
transporting $c$ once around this meridian, oriented by
$a_0\wedge c$.

\begin{lemma}[Comparison of the cusp models]
\label{lem:cusp-model-comparison}
Put $C_{\mathrm{lim}}=C^+(0)$.  After shrinking the cusp disc, the families
$C_\rho(t)=(1-\rho)C^+(t)+\rho C_{\mathrm{lim}}$, for
$0\leq\rho\leq1$, and, in real flat coordinates
$a,b\in\mathbb R^2$, the maps
$\zeta=(sB_0^++C_{\mathrm{lim}})a+b\longmapsto
\zeta-\sigma C_{\mathrm{lim}}a$, for $0\leq\sigma\leq1$,
define homotopies of the quotient torus family over the punctured disc.
They conjugate the corresponding deck actions and extend continuously
through the toric central fibre.  Each map preserves every toric face and
sends cosets of the corresponding torus isotropy subgroup to cosets of that
subgroup.  The extensions are $G$-equivariant.  The extension obtained from
$C_\rho(t)$ is the identity on the central fibre.  Consequently, the
restrictions of the
deformation retractions
supplied by
Lemma~\ref{lem:cusp-retraction} are homotopic on a nearby fibre for the actual
period block and for the model $sB_0^+$.
\end{lemma}

\begin{proof}
For the first family put
$B_{t,\rho}=B_0^+-2\pi\Im C_\rho(t)/\log|t|$.
Uniform invertibility of $B_{t,\rho}$, for $0<|t|<\epsilon$ and
$0\leq\rho\leq1$, follows after shrinking $\epsilon$ from the
invertibility of $B_0^+$.  On a nearby fibre every point has real flat
coordinates
$
 \zeta=(sB_0^++C^+(t))a+b
$
with $(a,b)\in\mathbb R^2\oplus\mathbb R^2$, unique modulo
$\mathbb Z^2\oplus\mathbb Z^2$ in the marked real torus.  Replacing
$(a,b)$ by another integral representative changes the image by the
corresponding period for $C_\rho(t)$.  Hence
\begin{equation}\label{eq:first-cusp-comparison}
 (sB_0^++C^+(t))a+b\longmapsto
 (sB_0^++C_\rho(t))a+b
\end{equation}
is independent of the representative and conjugates the full lattice and
deck actions.

Choose $a$ in a fixed compact fundamental parallelogram.  Since
$C_\rho(t)-C^+(t)=O(t)$, the torus multiplier in
\eqref{eq:first-cusp-comparison} tends uniformly to one as $t\to0$.  Its
logarithmic modulus divided by $\log|t|$ tends to
zero, so it changes neither the limiting fan face nor the corresponding
torus isotropy subgroup.  In every unimodular toric chart it therefore extends continuously
as the identity on each central orbit.  The extensions agree on chart
overlaps because \eqref{eq:first-cusp-comparison} is already a single map on the
dense quotient and the toric space is Hausdorff.

For the constant twist, put
\[
 m_\sigma(y)=\exp\!\left(
 -2\pi i\sigma C_{\mathrm{lim}}
 \bigl((B_0^+\otimes\mathbb R)^{-1}y\bigr)\right)
 \quad(y\in K_{\mathbb R}).
\]
The same deformation argument applies after retaining the finite-index
sublattice $L$; compare \citep[Lemmas~7.5--7.10]{Alpoge2026}.  On the
$K_{\mathbb R}$-cover of $T_B$, where $y=B_0^+a$, one has
\begin{equation}\label{eq:constant-comparison-equivariance}
 m_\sigma(y+B_0^+\lambda)e^{2\pi iC_{\mathrm{lim}}\lambda}
 =e^{2\pi i(1-\sigma)C_{\mathrm{lim}}\lambda}m_\sigma(y),
 \qquad \lambda\in\Gamma.
\end{equation}
Thus multiplication by $m_\sigma(y)$ conjugates the
$C_{\mathrm{lim}}$-twisted deck action to the action with twist
$(1-\sigma)C_{\mathrm{lim}}$, and commutes with the $K$-periods.

In the polar model of a toric chart, the positive part of $m_\sigma(y)$
stays in the same face and its compact part translates the phase.  If
$T_P\subset T_f$ is the isotropy subgroup of a face $P$, this
translation induces a translation of $T_f/T_P$; see
\citep[Proposition~3.1]{NakayamaOgus2010}.  Since $m_\sigma$ is continuous
and bounded on every closed dual cell, these maps extend over the toric
boundary and agree on common faces.  Their logarithmic displacement is
$o(|\log|t||)$, so the extension has the same limiting face as the
map on the punctured family.  Equation
\eqref{eq:constant-comparison-equivariance} also gives $G$-equivariance.
Taking $\rho=\sigma=1$ reduces the family to the
model with period block $sB_0^+$.
\end{proof}

\begin{lemma}[Deformation retraction of the cusp neighbourhood]\label{lem:cusp-retraction}
Let $W=f^{-1}(p_0)\subset A$ be the two-component central fibre determined
by the periodic triangulation $\mathcal T$.  For
a sufficiently small closed cusp disc, the following assertions hold.
\begin{enumerate}
\item The inclusion $W\hookrightarrow A$ is a strong deformation
retract.
\item The restriction of the deformation retraction to a nearby fibre is
homotopic to the collapse map
$c:T_B\times T_f\to W$.
Over the interior of a dual two-cell it leaves $T_f$ unchanged; over the
dual edge corresponding to an edge of $\mathcal T$ with primitive
direction $n$, it quotients by the circle $S_n\subset T_f$ generated by
$n$; over a dual vertex it collapses all of $T_f$.
\item The inclusion of the boundary induces
\begin{equation}\label{eq:cusp-pi1-map}
 \pi_1(M)\longrightarrow\pi_1(A)=\mathbb Z\langle \widehat\gamma,u\rangle,
 \qquad
 \widehat\gamma\mapsto\widehat\gamma,\quad
 u\mapsto u,\quad w,d,a_0\mapsto1.
\end{equation}
\item The suspension tori $\xi_w=\operatorname{sw}_{w}(a_0)$ and
$\xi_d=\operatorname{sw}_{d}(a_0)$
bound embedded solid tori in $A$.
\end{enumerate}
\end{lemma}

\begin{proof}
\smallskip
\noindent\emph{A semistable finite cover and equivariant shrinking.}
The finite cover allows the retraction to be constructed on the closed
neighbourhood while retaining the map induced by its boundary inclusion.
In the fan of cones over $\mathcal T$, pass
from $\Gamma$ to $\Gamma'=2\Gamma$.  Its translation image is
$L'=2L$, and the quotient deck group is $G$.  Every difference of two
vertices
of a simplex of $\mathcal T$ is primitive in $K$, whereas $L'\subset2K$.
Thus no edge or triangle identifies two of its vertices modulo $L'$.
The central fibre of the $\Gamma'$-quotient is consequently a reduced
simple normal crossings divisor, with local equation
$t=z_1\cdots z_r$, where $1\le r\le3$.
The total space is smooth, the map is proper and flat (locally, $t$ is a
non-zero-divisor in a regular local ring), and it is smooth over the
punctured disc.  The $G$-action is free even on the double and
triple strata: if $\bar\lambda\in G$ fixed $[x]$, then
$\Psi_\lambda x=\Psi_\mu x$ for some $\mu\in\Gamma'$; freeness of the
original $\Gamma$-action gives $\lambda=\mu$.

The finite cover is therefore a semistable degeneration of the kind used by
Usui: its total space is smooth, the map to the disc is proper and flat, the
central fibre is reduced with simple normal crossings, and the punctured
fibres are smooth.  We take the inverse image of a closed subdisc preserved
by the radial action.

We apply Usui's monoid action and shrinking map
\citep[Theorem~5.2 and equations (5.2.5.2), (5.2.5.4), and (5.2.12)]{Usui2001}.
We make the auxiliary choices in Usui's construction $G$-equivariant as
follows.  Index the
irreducible components upstairs by a finite set $J$, which $G$ permutes.
Choose a finite $G$-stable system of standard-coordinate charts, each with
trivial stabilizer, and transport coordinates and branches around every
$G$-orbit.  Run the descending inductions of Usui's
Propositions~3.2 and~4.3 simultaneously on the $G$-orbits of subsets of
$J$.  Whenever an indexed partition or cutoff family
$\{p_U\}$ occurs, replace it by
$p_U^G(x)=|G|^{-1}\sum_{h\in G}p_{h^{-1}U}(h^{-1}x)$.
Then $p_{gU}^G(gx)=p_U^G(x)$; subordination, closed-support conditions,
and the fibre-constancy conditions (3.2.22), (3.2.29), and (4.3.4) are
preserved because they are affine and are permuted by $G$.  Induction in
the defining formulas therefore gives $\pi_{gI}\circ g=g\circ\pi_I$ and
$z_{gi}(gy)=z_i(y)$
as systems of branches.  Thus the normal projections and global equations
are $G$-equivariant.

All central multiplicities are one.  Coordinate permutations preserve the
cube, its face projections, and hence Usui's action $R$ on hyperbolic
polar coordinates.  Write $z_i=r_iu_i$, and let $\xi$ be a point of the
logarithmic fibre over $y=\tau_X(\xi)$.  For the radial submonoid,
equation (5.2.12) reads
\begin{equation}\label{eq:usui-radial-phases}
 r_i((s,1)\cdot\xi)=R(s,y)_i,
 \qquad u_i((s,1)\cdot\xi)=u_i(\xi).
\end{equation}
For $(s,1)$, the phase coordinates in
\eqref{eq:usui-radial-phases} are unchanged.  Since $G$ permutes the chosen
coordinates, the radial action is strictly $G$-equivariant.  The homotopy
$H'(x,\rho)=((1-\rho),1)\cdot x$ stays in this closed neighbourhood by (5.2.5.2),
lands in the central fibre at $\rho=1$, and fixes that fibre pointwise by
(5.2.5.4) and \eqref{eq:usui-radial-phases}.  Let $q:A'\to A=A'/G$ be the
finite quotient of this neighbourhood.  Since $q\times\mathrm{id}_{[0,1]}$ is
a quotient map, there is a unique continuous homotopy
$H:A\times[0,1]\to A$ satisfying
$H(q(x),\rho)=q\bigl(H'(x,\rho)\bigr)$.
It fixes $W$ pointwise and has image in $W$ at $\rho=1$, so it is a
strong deformation retraction, proving the first assertion.

\smallskip
\noindent\emph{The collapse map on a nearby fibre.}
The polar-coordinate description gives the following model for the descended
central fibre.  For $b\in T_B$, let $P_b$ be the face of the
triangulation $\mathcal T$ whose open dual cell contains $b$.  If
$P_b$ has vertices $v_0,\ldots,v_r$, let $T_b\subset T_f$ be the
subtorus whose cocharacter lattice is generated by
$v_1-v_0,\ldots,v_r-v_0$.  Define
$(b,x)\sim(b',x')$ precisely when $b=b'$ and $x'-x\in T_b$.  Then
\begin{equation}\label{eq:cusp-collapse-quotient}
 W=(T_B\times T_f)/{\sim}.
\end{equation}
Thus $T_b$ is trivial when $P_b$ is a vertex, is $S_n$ when $P_b$
is an edge of primitive direction $n$, and is all of $T_f$ when
$P_b$ is a unimodular triangle.  These are the three cases in
Lemma~\ref{lem:cusp-retraction}.  The same polar description is the local
model of \citep[Proposition~3.1 and Theorems~3.7, 5.1]{NakayamaOgus2010}.

To identify this map, work first over a sufficiently small positive real
value of $t$.  In the chart belonging to
$P=[v_0,\ldots,v_r]$, write $t=z_0\cdots z_r$ and
$z_i=r_iu_i$.  Equation~\eqref{eq:usui-radial-phases} says that the
radial action changes only the radii.  At its endpoint, projection from the
logarithmic chart to the ordinary toric chart forgets precisely the phases
of the vanishing coordinates.  On the fixed-$t$ fibre their product phase
is fixed, so the endpoint fibre is
\[
 \ker\bigl((S^1)^{r+1}\longrightarrow S^1\bigr)\cong(S^1)^r,
 \qquad
 \mathbb Z\langle(v_i,1):0\leq i\leq r\rangle\cap(K\oplus0)
 =\mathbb Z\langle(v_i-v_0,0):1\leq i\leq r\rangle.
\]
The nonvanishing-coordinate phases are retained.  The description is
invariant under monomial coordinate changes; the local descriptions agree on
overlaps by \citep[p.~30, Claim~2 following (5.2.12)]{Usui2001}.  The radial endpoint
induces a map
$h:T_B\to T_B$ preserving every closed dual cell.  Indeed, a closed dual
cell is characterized in a toric chart by the set of coordinates allowed to
vanish.  Usui's radial formula changes only the corresponding radii and its
face projections commute on inclusions; it neither introduces a coordinate
 outside that face nor changes the formula on a common face.  Hence
$h(\overline P)\subset\overline P$ for every dual cell $P$, including
its lower-dimensional boundary.  In particular, $h$ fixes every dual
vertex.  Let $p:K_{\mathbb R}\to T_B=K_{\mathbb R}/L$ be the universal
covering.  Fix a lift $\widetilde v$ of a dual vertex, and let $\widetilde h$
be the lift of $h\circ p$ fixing it.

On the one-skeleton, $h$
maps each edge into itself relative to its endpoints, so $h_*$ is the
identity on $\pi_1(T_B)$.  It follows that
$\widetilde h(y+\ell)=\widetilde h(y)+\ell$ for
$y\in K_{\mathbb R}$ and $\ell\in L$.
Lifting successively across adjacent cells now shows that
$\widetilde h$ maps each closed lifted cell into itself.
For $0\leq u\leq1$, the explicit interpolation
$\widetilde h_u(y)=(1-u)\widetilde h(y)+uy$
is still $L$-equivariant and stays in the same lifted cell, since that cell
is convex.  This translation-equivariant interpolation is compatible on common faces;
because the isotropy subgroups above are constant on open faces and only
increase on their boundaries, it lifts to a homotopy of the quotient maps.
Thus the restriction of the shrinking map for the model with period block
$sB_0^+$ is homotopic to $c$.

Lemma~\ref{lem:cusp-model-comparison} now removes the holomorphic and constant
deck twists without changing the torus isotropy subgroups.  Hence the
restriction of the actual deformation retraction to $F$ is homotopic to
the collapse map $c$.

\smallskip
\noindent\emph{The fundamental group and the boundary inclusion.}
Projection of \eqref{eq:cusp-collapse-quotient} to $T_B$ has the section
$b\mapsto[b,1]$.  Every loop in $T_f$ can be moved to a dual vertex,
where all of $T_f$ is collapsed.  Cellular van Kampen therefore gives
$\pi_1(A)=\pi_1(T_B)=L$.
Under the identification in \eqref{eq:cusp-B-marking}, its generators
correspond to $\widehat\gamma,u$; the loops $w,d$ in $T_f$ are
null-homotopic.

For the meridian and the suspension tori, use the chart of the
triangle $[0,e_1,e_2]$, use $t=z_0z_1z_2$, $x_1=z_1$, and
$x_2=z_2$.
Here the coordinate meridian is counterclockwise, whereas $a_0$ is
clockwise.  The three coordinate discs are
\[
 \begin{split}
 E_0&=\{z_1=z_2=1,\ |z_0|\le\eta\},\\
 E_w&=\{z_0=\eta,\ z_2=1,\ |z_1|\le1\},\\
 E_d&=\{z_0=\eta,\ z_1=1,\ |z_2|\le1\}.
 \end{split}
\]
Their counterclockwise boundary classes in $H_1(M;\mathbb Z)$ are
\begin{center}
\begin{tabular}{c@{\qquad}c}
\toprule
disc & class in $H_1(M;\mathbb Z)$\\
\midrule
$E_0$ & $-a_0$\\
$E_w$ & $-a_0+w$\\
$E_d$ & $-a_0+d$\\
\bottomrule
\end{tabular}
\end{center}
These discs prove that $a_0,w,d$ die, and hence establish
\eqref{eq:cusp-pi1-map}.

\smallskip
\noindent\emph{The suspension tori.}
Set $V_w=\{z_2=1,\ |z_1|=1,\ |z_0|\le\eta\}$ and
$V_d=\{z_1=1,\ |z_2|=1,\ |z_0|\le\eta\}$.
Each is an embedded solid torus whose boundary lies in $M$.
For $V_w$, if $\phi=\arg z_0$ and $\psi=\arg z_1$, then on the
boundary $(\arg t,\arg x_1,\arg x_2)=(\phi+\psi,\psi,0)$.
The determinant-one change $(\alpha,\beta)=(\phi+\psi,\psi)$ gives
$[\partial V_w]=(-a_0)\wedge(-a_0+w)
=-a_0\wedge w=\pm[\xi_w]$.
The same calculation gives
$[\partial V_d]=-a_0\wedge d=\pm[\xi_d]$.
All displayed coordinate sets project injectively.  Two points of
$E_0,V_w,V_d$ that differ by a deck transformation have the same normalized
logarithmic coordinate $y$, so \eqref{eq:an-log-coordinate-translation} gives
$B_t\lambda=0$.  The matrix $B_t$ tends to $B_0^+$ and is invertible on the
chosen cusp disc, and hence
$\lambda=0$.  On $E_w$ and $E_d$, respectively, the identities
$t=\eta z_1$ and $t=\eta z_2$ first force equality of the two points;
freeness then forces the deck element to be trivial.  On the central
strata, $D_v^\circ$ is sent to $D_{v+B_0^+\lambda}^\circ$, and
injectivity of $B_0^+$ again gives $\lambda=0$.  Since the quotient is
locally biholomorphic, the compact injective images are embedded discs and
solid tori downstairs.
\end{proof}

The cokernel term in the cohomological Wang sequence is already determined by
\eqref{eq:cusp-monodromy-cohomology}:
$C_0:=\operatorname{coker}(T_0-I:V^+\to V^+)
=\mathbb Z[w^\vee]\oplus\mathbb Z[\Delta]
\oplus(\mathbb Z/2)[\gamma]$.
The solid-torus bounds have a cohomological consequence.  Pairing with
$\xi_w$ and $\xi_d$ defines two homomorphisms
$H^2(M;\mathbb Z)\to\mathbb Z$.  On the injected copy of $C_0$ in the
Wang sequence, these homomorphisms extract, with the chosen orientations,
the $[w^\vee]$- and $[\Delta]$-coefficients, respectively; they vanish on
every monodromy-invariant two-cocycle on the fibre.  Since both suspension
tori bound in $A$, the two evaluations vanish on every class restricted from
$H^2(A;\mathbb Z)$.  This formulation does not require a splitting of the
Wang sequence.

\subsection{Homology of the cusp fibre}

The calculation below follows the cellular calculation in
\citep[Appendix~A]{Alpoge2026}, using the quotient dual graph for the
index-two lattice.  The middle panel of
Figure~\ref{fig:cusp-combinatorics} records the dual graph and its translation
labels.
Modulo horizontal translation, the lower and upper triangles at height
$j\in\mathbb Z/2$ give four dual zero-cells
$\mathsf L_j,\mathsf U_j$.  Orient the six dual edges by
$d_j,\upsilon_j:\mathsf L_j\to\mathsf U_j$ and
$h_j:\mathsf L_j\to\mathsf U_{j-1}$.
In the edge order $(d_0,\upsilon_0,h_0,d_1,\upsilon_1,h_1)$, their
translation labels are
$(0,-e_1,-2e_2,0,-e_1,0)$.  There are two dual two-cells
$P_0,P_1$, with $\partial P_0=-d_0+\upsilon_0+d_1-\upsilon_1$ and
$\partial P_1=-\partial P_0$.
Thus the ordinary cellular differentials of the dual torus are
\begin{equation}\label{eq:dual-torus-differentials}
 \partial_1=
 \begin{pmatrix}
 -1&-1&-1& 0& 0& 0\\
  1& 1& 0& 0& 0& 1\\
  0& 0& 0&-1&-1&-1\\
  0& 0& 1& 1& 1& 0
 \end{pmatrix},
 \qquad
 \partial_2=
 \begin{pmatrix}
 -1& 1\\ 1&-1\\ 0&0\\ 1&-1\\-1&1\\0&0
 \end{pmatrix}.
\end{equation}
The nonzero Smith invariants of $\partial_1$ are $1,1,1$, and that of
$\partial_2$ is $1$.  Moreover $\operatorname{im}\partial_2$ is the
primitive rank-one subgroup generated by
$-d_0+\upsilon_0+d_1-\upsilon_1$ inside $\ker\partial_1$.  Thus the
bottom row has homology $(\mathbb Z,\mathbb Z^2,\mathbb Z)$; these Smith
invariants also prevent torsion from entering this row of the filtration
spectral sequence.

For a dual cell $\sigma$, write $T_\sigma=T_b$ for $b$ in its interior,
and filter $W$ by the inverse images, under the projection $W\to T_B$, of
the skeleta of the dual complex.  Over $\sigma$, this projection has fibre
$T_f/T_\sigma$, and the filtration spectral sequence has
$E^1_{p,q}=\bigoplus_{\sigma^p}H_q(T_f/T_\sigma;\mathbb Z)$.
\begin{samepage}
Its nonzero terms are
\begin{center}
\begin{tabular}{c@{\quad}ccc}
\toprule
 & $p=0$ & $p=1$ & $p=2$\\
\midrule
$q=2$ & 0 & 0 & $\mathbb Z^2$\\
$q=1$ & 0 & $\mathbb Z^6$ & $K\oplus K$\\
$q=0$ & $\mathbb Z^4$ & $\mathbb Z^6$ & $\mathbb Z^2$\\
\bottomrule
\end{tabular}
\end{center}
\end{samepage}
The bottom differential is \eqref{eq:dual-torus-differentials}.  Put
$n_d=e_2-e_1$, $n_v=e_2$, and $\nu_n(k)=\det(n,k)$, and write
$\nu_d=\nu_{n_d}$ and $\nu_v=\nu_{n_v}$.
The only differential in the $q=1$ row is
\begin{equation*}
 \begin{split}
 \delta:K\oplus K&\longrightarrow\mathbb Z^6,\\
 (k,k')&\longmapsto
 \bigl(-\nu_d(k-k'),\nu_v(k-k'),0,
       \nu_d(k-k'),-\nu_v(k-k'),0\bigr).
 \end{split}
\end{equation*}
The zero horizontal entries result from the two opposite incidences of the
same horizontal edge orbit.  The map
$K\to\mathbb Z^2$, $k\mapsto(\nu_d(k),\nu_v(k))$, is unimodular.
Therefore $\ker\delta=\{(k,k):k\in K\}$, and its cokernel is the free
abelian group
$\mathbb Z\langle\bar h_0,\bar h_1,\bar d,\bar\upsilon\rangle$, where
$\bar d=[d_0]=[d_1]$ and
$\bar\upsilon=[\upsilon_0]=[\upsilon_1]$.
Equivalently, the two nonzero Smith invariants of $\delta$ are both one.
Every possible higher differential has zero target.  Consequently there
are no torsion or additive extension ambiguities, and
\begin{equation}\label{eq:cusp-homology}
 H_k(A;\mathbb Z)=H_k(W;\mathbb Z)=
 (\mathbb Z,\mathbb Z^2,\mathbb Z^5,
   \mathbb Z^2,\mathbb Z^2),\qquad 0\le k\le4.
\end{equation}

In degree two, the two
base cycles are $\beta_1=d_0-\upsilon_0$ and
$\beta_2=d_0+d_1-h_0-h_1$; their translation labels are
$e_1$ and $2e_2$, respectively.  If
$\beta=\sum m_e e$ and $v\in K$, the collapse over an edge gives
$c_*(\beta\otimes v)=\sum_e m_e\nu_{n_e}(v)[C_e]$.
Here $C_e$ is the oriented surviving circle in $T_f/S_{n_e}$ over the
dual edge $e$.
With compatible orientations, the four mixed generators have images
\begin{equation}\label{eq:specialization-table}
\begin{array}{c|cc}
 &e_1&e_2\\ \hline
\beta_1&-\bar d+\bar\upsilon&-\bar d\\
\beta_2&-2\bar d&-2\bar d-\bar h_0-\bar h_1.
\end{array}
\end{equation}
The base orientation class maps to a primitive class $\omega_B$.  The phase
orientation class maps to zero because it may be represented over a dual
vertex, where all of $T_f$ is collapsed.  It follows directly from
\eqref{eq:specialization-table} that
\begin{equation}\label{eq:specialization-image}
 c_*H_2(F;\mathbb Z)
 =\mathbb Z\langle\omega_B,\bar d,\bar\upsilon,
                  \bar h_0+\bar h_1\rangle.
\end{equation}
This subgroup is saturated.  Its cokernel is infinite cyclic, primitively
generated by $\bar h_0$, with $\bar h_1=-\bar h_0$; in particular,
$\bar h_0-\bar h_1$ is twice, rather than once, a primitive generator.
With target basis
$(\omega_B,\bar h_0,\bar h_1,\bar d,\bar\upsilon)$ and source basis
$([T_B],[T_f],\beta_1e_1,\beta_1e_2,\beta_2e_1,\beta_2e_2)$, the
specialization matrix is
\begin{equation*}
 \begin{pmatrix}
 1&0& 0& 0& 0& 0\\
 0&0& 0& 0& 0&-1\\
 0&0& 0& 0& 0&-1\\
 0&0&-1&-1&-2&-2\\
 0&0& 1& 0& 0& 0
 \end{pmatrix}.
\end{equation*}
Its four nonzero Smith invariants are all one, so the image in
\eqref{eq:specialization-image} is saturated.

On cohomology put
\begin{equation*}
 I_0=(\textstyle\bigwedge^2V^+)^{T_0}
 =\mathbb Z\langle
 \gamma u^\vee,\gamma\Delta,u^\vee w^\vee,
 u^\vee\Delta-2\gamma w^\vee\rangle.
\end{equation*}
Via $A\simeq W$, restriction to a nearby fibre is the pullback $c^*$.
Since the relevant homology groups are free, its matrix is the transpose of
the specialization matrix above.  Its image is therefore saturated of rank
four, and its kernel is primitive of rank one.  The image lies in $I_0$ by
monodromy invariance; since both are saturated rank-four sublattices of
$\bigwedge^2V^+$ with the same rational span, they are equal.  Thus
\begin{equation}\label{eq:cusp-H2-splitting}
 0\longrightarrow\mathbb Z
 \longrightarrow H^2(A;\mathbb Z)
 \xrightarrow{\ \operatorname{res}_F\ } I_0\longrightarrow0.
\end{equation}
The sequence splits as groups, but no splitting will be used yet.  Meanwhile
the Wang sequence and \eqref{eq:cusp-monodromy-cohomology} give
\begin{equation}\label{eq:cusp-Wang}
 0\longrightarrow C_0\longrightarrow H^2(M;\mathbb Z)
 \longrightarrow I_0\longrightarrow0,
 \qquad
 C_0=\mathbb Z[w^\vee]\oplus\mathbb Z[\Delta]
       \oplus(\mathbb Z/2)[\gamma].
\end{equation}
Write $\vartheta=[\gamma]$ for the unique nonzero torsion element of $C_0$.

\begin{lemma}[The cusp double curve]
\label{lem:cusp-double-curve}
Let $D$ be either component of the cusp fibre, let
$\nu:S\to D$ be its normalization, and let
$C_{\mathrm{dbl}}\cong\mathbb P^1$ be its double curve.  Number the boundary
curves of $S\cong\dP$ as in Figure~\ref{fig:cusp-combinatorics}, so that
$C_1,C_4$ are the two conductor branches, and put
$C=C_2+C_3+C_5+C_6$.  Then
\begin{equation}\label{eq:normalization-normal-bundle}
 \nu^*(\mathcal O_Y(D)|_D)=\mathcal O_S(-C).
\end{equation}
Moreover, $D^3=0$, $D\cdot C_{\mathrm{dbl}}=-2$, and
$\deg N_{C_{\mathrm{dbl}}/Y}=-2$.
\end{lemma}

\begin{proof}
Let $D'$ be the other cusp component.  Since $Y$ is smooth, both
components are Cartier.  On a holomorphic neighbourhood of the cusp fibre,
$D+D'=\operatorname{div}(t)$.  The pullback of $D'|_D$ to the
normalization is the reduced divisor $C$, which proves
\eqref{eq:normalization-normal-bundle}.  Each $C_i$ is a $(-1)$-curve,
and the only adjacent pairs among the four components of $C$ are
$(C_2,C_3)$ and $(C_5,C_6)$.  Hence
$C^2=-4+2(1+1)=0$.  The projection formula for the finite normalization
therefore gives
$D^3=\int_S c_1(\mathcal O_S(-C))^2=C^2=0$.
Either conductor branch, say $C_1$, meets $C$ in two reduced points and
maps isomorphically to $C_{\mathrm{dbl}}$.  Restricting
\eqref{eq:normalization-normal-bundle} to $C_1$ gives
$D\cdot C_{\mathrm{dbl}}=-2$.  Locally along the double curve the two
branches of $D$ have product equation $z_1z_2=0$.  The product of their
normal lines is both $\det N_{C_{\mathrm{dbl}}/Y}$ and
$\mathcal O_Y(D)|_{C_{\mathrm{dbl}}}$, proving the last equality.
\end{proof}

\begin{proposition}[The class of a cusp component]
\label{prop:cusp-component-class}
There is a topological line bundle $\mathcal E$ on the closed cusp neighbourhood
whose first Chern class $\kappa=c_1(\mathcal E)$ generates the kernel in
\eqref{eq:cusp-H2-splitting} and satisfies
\begin{equation}\label{eq:cusp-component-class}
 [D_0]=2\kappa,\qquad [D_1]=-2\kappa.
\end{equation}
Its restriction to a smooth fibre has zero Chern class, while on the boundary
mapping torus
\begin{equation}\label{eq:kappa-boundary}
 \kappa|_M=\vartheta=[\gamma].
\end{equation}
Thus $\mathcal E$ is topologically trivial on every smooth fibre, although
its restriction to $M$ is not.
\end{proposition}

\begin{proof}
Let $D_{(a,j)}$ denote the height-one ray divisor on the infinite toric
cover and set
\[
 E=\sum_{(a,j)\in\mathbb Z^2}
   -\left\lceil\frac j2\right\rceil D_{(a,j)}.
\]
The fan is locally finite, and only the finitely many ray divisors belonging
to a fixed cone meet its affine chart.  Hence this sum is a locally finite
torus-invariant Weil divisor.  The toric cover is smooth, so its local rings
are factorial and $E$ is Cartier on every chart.  The ratios of the local
equations on overlaps are units and therefore define $\mathcal O(E)$.

Let $D_{\mathrm{ev}}$ be the sum over even $j$, and denote by $D_0$
and $D_1$ the components obtained from the divisors with even and odd
values of $j$, respectively.  For the character
$m=(0,1,1)$, evaluation on every height-one ray gives
\begin{equation}\label{eq:half-divisor-identity}
 2E=D_{\mathrm{ev}}-\operatorname{div}(\chi^m).
\end{equation}
Now write $\lambda=r_\lambda\widehat\gamma+s_\lambda u$.  Its height-one translation
is $(s_\lambda,-2r_\lambda)$, and the torus-translation factor fixes every
invariant ray divisor.  Coefficientwise comparison therefore gives
$\Psi_\lambda^*E=E+r_\lambda\operatorname{div}(t)$.
Multiplication by $t^{r_\lambda}$ gives the regular isomorphism
$\Psi_\lambda^*\mathcal O(E)
=\mathcal O(E+r_\lambda\operatorname{div}(t))
\longrightarrow\mathcal O(E)$.
This remains regular when $r_\lambda<0$: the pole of the rational
multiplier is exactly cancelled by the change of the divisor sheaf in the
source.  Since $r_{\lambda+\mu}=r_\lambda+r_\mu$ and
$\Psi_\lambda^*t=t$, these~isomorphisms satisfy the strict cocycle
condition.  They therefore descend $\mathcal O(E)$ to a line bundle
$\mathcal E$ on the cusp neighbourhood.

Pulling the character $\chi^m$ through the holomorphic torus translation
introduces a nowhere-vanishing holomorphic unit $u_\lambda(t)$.  The deck
law makes $\{u_\lambda\}$ a multiplicative cocycle.  On a sufficiently
small disc, choose holomorphic logarithms on the two generators of
$\Gamma$ and extend them additively.  The cocycles
$\exp(\sigma\log u_\lambda)$, $0\leq\sigma\leq1$, give a homotopy from
this unit cocycle to the trivial one.  Thus it contributes no integral first
Chern class, and the descended form of
\eqref{eq:half-divisor-identity} yields $[D_0]=2\kappa$.  Since
$D_0+D_1=\operatorname{div}(t)$ on the cusp neighbourhood, the second equality in
\eqref{eq:cusp-component-class} follows.  Restriction to a nearby fibre
leaves only constant character multipliers, so $c_1(\mathcal E|_F)=0$.

On $t=\eta e^{2\pi i\theta}$, the positive real factor
$\eta^{r_\lambda}$ is homotopically trivial and the descent multiplier is
$e^{2\pi i\theta r_\lambda}=e^{2\pi i\theta\gamma(\lambda)}$.
In the Wang sequence \eqref{eq:cusp-Wang}, this cocycle represents
$[\gamma]\in C_0$, and hence proves
\eqref{eq:kappa-boundary}.

Finally, let $C_{\mathrm{dbl}}$ be the double curve obtained by identifying
the two horizontal boundary curves in the normalization of $D_0$.
Lemma~\ref{lem:cusp-double-curve} gives
$D_0\cdot C_{\mathrm{dbl}}=-2$.  Therefore
\eqref{eq:cusp-component-class} gives
$\kappa(C_{\mathrm{dbl}})=-1$.  Thus $\kappa$ is primitive.  Since its
restriction to $F$ vanishes, it generates the kernel in
\eqref{eq:cusp-H2-splitting}.
\end{proof}

\subsection{The two multiple fibres}

The image of $Y_{\mathrm{fin}}$ in the base is a closed disc containing
the orbifold points $p_1$ and $p_2$.  Split it into two closed discs
$U_1,U_2$ meeting in a band, with $p_j$ in
$U_j$, and put $Y_j=f^{-1}(U_j)$.  On the orbifold cover of
$U_j$, the homotopy
$R_\rho(s_j,x)=((1-\rho)s_j,x)$, for $0\leq\rho\leq1$,
commutes with
$(s_j,x)\mapsto(\zeta_js_j,A_jx+v_j/m_j)$, where
$\zeta_j=e^{-2\pi i/m_j}$.  It therefore descends to a
strong deformation retraction of $Y_j$ onto its reduced central fibre
$S_j$.  Hence
$Y_j\simeq S_j$.  After trivializing the torus bundle over the overlap
band, $Y_1\cap Y_2\simeq F\times[0,1]\searrow F$.
This retraction is compatible with the two retractions $Y_j\searrow S_j$,
and the induced attaching maps are the coverings
$\pi_j:F\to S_j$.  Thus $Y_{\mathrm{fin}}$ has the homotopy-pushout
model $S_1\cup_F S_2$.  Define
$\eta_1=4u^\vee+2w^\vee+3\Delta$,
$\eta_2=u^\vee+w^\vee+\Delta$, and
$Q=u^\vee w^\vee+3\gamma\Delta$.  Here $T_j$ denotes the
contragredient action fixed at the beginning of this section.
The affine Bieberbach group of the reduced fibre is
\[
 \mathcal G_j=
 \left\langle\Lambda^+,r_j\ \middle|\
 r_j\lambda r_j^{-1}=A_j\lambda,\quad r_j^{m_j}=v_j
 \right\rangle,
 \qquad (m_1,m_2)=(3,4),
\]
and fits into
$1\to\Lambda^+\to\mathcal G_j\to C_{m_j}\to1$.
The spectral-sequence calculation below follows the method of
\citep[Appendix~A, Lemma~A.4]{Alpoge2026}, with $\Lambda^+$ in place of
$\Lambda$; in particular, the extension associated with $S_2$ changes.
For the free affine action on $F$, the Cartan--Leray spectral sequence,
equivalently the Lyndon--Hochschild--Serre sequence of this extension, is
\begin{equation*}
 E_2^{p,q}=H^p\!\left(C_{m_j};\textstyle\bigwedge^qV^+\right)
 \Longrightarrow H^{p+q}(S_j;\mathbb Z).
\end{equation*}
The edge map is the covering pullback $\pi_j^*$.  In degree one, an
invariant character $\chi$ descends exactly when
$\chi(v_j)\equiv0\pmod{m_j}$.  Hence
$\pi_1^*H^1(S_1;\mathbb Z)=
\mathbb Z\langle3\gamma,\eta_1\rangle$ and
$\pi_2^*H^1(S_2;\mathbb Z)=
\mathbb Z\langle4\gamma,\eta_2\rangle$.

We determine the images of $\pi_j^*$ on $H^2$ without using a
transgression formula.  Transfer identifies $H^2(S_j;\mathbb Q)$ with
$(\bigwedge^2V^+_{\mathbb Q})^{T_j}$.  Direct calculation from the matrices
$A_j$ gives the integral invariant lattices
$I_1=\mathbb Z\langle \gamma\eta_1,Q\rangle$ and
$I_2=\mathbb Z\langle \gamma\eta_2,Q\rangle$.
With the orientation
$\gamma u^\vee w^\vee\Delta$, their Gram matrices in the displayed bases
are
\[
 \begin{pmatrix}0&3\\3&6\end{pmatrix},
 \qquad
 \begin{pmatrix}0&1\\1&6\end{pmatrix},
\]
of determinants $-9$ and $-1$, respectively.  Put
$L_j=\pi_j^*H^2(S_j;\mathbb Z)_{\mathrm{free}}\subset I_j$.
The projection formula gives
$\langle\pi_j^*x\smile\pi_j^*y,[F]\rangle
=m_j\langle x\smile y,[S_j]\rangle$.
The intersection form on $H^2(S_j;\mathbb Z)_{\mathrm{free}}$ is
unimodular, so $|\det L_j|=m_j^2$.  It follows that
$[I_1:L_1]=1$ and $[I_2:L_2]=4$, and in particular
$L_1=\mathbb Z\langle \gamma\eta_1,Q\rangle$.

To identify the index-four sublattice $L_2$, put
$k_1=u-d$ and $k_2=w-d$.  At $p_2$, the monodromy and translation
vector satisfy $A_2k_1=k_2$, $A_2k_2=-k_1$, $A_2d=d$, and
$v_2=-\widehat\gamma-3k_1+3k_2$.
Eliminating $\widehat\gamma$ with $r_2^4=v_2$ gives
\begin{equation}\label{eq:order-four-product-model}
 \pi_1(S_2)=
 \bigl(\mathbb Z^2\langle k_1,k_2\rangle
       \rtimes_R\mathbb Z\langle r_2\rangle\bigr)
 \times\mathbb Z\langle d\rangle,
 \qquad R(k_1)=k_2,\quad R(k_2)=-k_1.
\end{equation}
Both $S_2$ and $N_R\times S^1_d$ are aspherical: the former is a free
finite affine quotient of $F$, while the latter has universal cover
$\mathbb R^4$.  The displayed fundamental-group identification therefore
yields $S_2\simeq N_R\times S^1_d$, where $N_R$ is the mapping torus of
the quarter-turn $R$ on the two-torus.  We write an element of the displayed
semidirect product in the normal form $(v,n;m)$, meaning
$v r_2^n d^m$, with $v\in\mathbb Z^2\langle k_1,k_2\rangle$.
In these coordinates
$\widehat\gamma=(-3k_1+3k_2,-4;0)$.
The Wang and K\"unneth sequences give two generators for the free part of
$H^2(N_R\times S^1_d;\mathbb Z)$: the fibre-orientation class on $N_R$
and the product of the mapping-torus base class with the $d$-circle class.
The first pulls back to
$(u^\vee-3\gamma)(w^\vee+3\gamma)=Q-3\gamma\eta_2$.  For the second, the
normal-form expression for $\widehat\gamma$ shows that the base class pulls
back to $-4\gamma$, while the $d$-circle class pulls back to $\eta_2$.
Thus, up to signs, the two pullback classes are
$Q-3\gamma\eta_2$ and $4\gamma\eta_2$.
Consequently
$L_2=\mathbb Z\langle4\gamma\eta_2,Q-3\gamma\eta_2\rangle
=\mathbb Z\langle4\gamma\eta_2,\gamma\eta_2+Q\rangle$, and the pullback lattices are
\begin{equation*}
\begin{array}{c|c|c}
 &\pi_j^*H^1(S_j;\mathbb Z)
 &\pi_j^*H^2(S_j;\mathbb Z)_{\mathrm{free}}\\ \hline
j=1&\mathbb Z\langle3\gamma,\eta_1\rangle
   &L_1=\mathbb Z\langle \gamma\eta_1,Q\rangle\\
j=2&\mathbb Z\langle4\gamma,\eta_2\rangle
   &L_2=\mathbb Z\langle4\gamma\eta_2,\gamma\eta_2+Q\rangle.
\end{array}
\end{equation*}
Comparing coefficients gives
\begin{equation}\label{eq:finite-fibre-lattice-arithmetic}
 L_1\cap L_2=4\mathbb ZQ,
 \qquad
 L_1+L_2=\mathbb Z\langle\gamma\eta_1,\gamma\eta_2,Q\rangle.
\end{equation}
The latter sum is saturated: on the coordinates
$(\gamma w^\vee,\gamma\Delta,u^\vee w^\vee)$, the three displayed generators have a
minor of determinant $-1$.

We first construct the torsion class on $S_2$.  The homological
coinvariants are
$\Lambda^+/(A_2-I)\Lambda^+=\mathbb Z[\widehat\gamma]\oplus\mathbb Z[u]
\oplus(\mathbb Z/2)[d-u]$.  Since the affine relation in these coinvariants
is $4r_2=v_2=-\widehat\gamma$, eliminating $\widehat\gamma$ gives
$H_1(S_2;\mathbb Z)=\mathbb Z[r_2]\oplus\mathbb Z[u]
\oplus(\mathbb Z/2)[d-u]$.  Thus the universal coefficient theorem gives a
unique nonzero torsion class
$\alpha_2\in H^2(S_2;\mathbb Z)$, with
$\operatorname{Tor}H^2(S_2;\mathbb Z)\cong\mathbb Z/2$.

To determine its boundary image, consider the normal line of $S_2$.  The
affine generator $r_2$ acts on this line by the unitary
character $r_2\mapsto e^{-2\pi i/4}$, which is trivial on $\Lambda^+$.
This character has the real lift
$\widetilde\chi(\lambda)=\gamma(\lambda)$ and
$\widetilde\chi(r_2)=-1/4$,
because $\gamma\circ A_2=\gamma$ and
$4\widetilde\chi(r_2)=-1=\gamma(v_2)$.  Thus the normal line is
topologically trivial.  In the exponential sequence
$0\to\mathbb Z\to\mathbb R\to U(1)\to0$, the real character
$\widetilde\chi$ lifts the unitary normal character, so its Bockstein,
which is the first Chern class of the normal line, vanishes.  Let
$p:M_2=\partial Y_2\to S_2$ be its unit-circle bundle.  Its Euler class is
therefore zero, and the Gysin sequence gives
$\ker(p^*:H^q(S_2;\mathbb Z)\to H^q(M_2;\mathbb Z))
=\operatorname{im}({-}\smile c_1(p))=0$.
Thus $p^*$ is injective in every degree.

The relations in the cokernel term of the cohomological Wang sequence are
$6\gamma-u^\vee+w^\vee=0$, $-u^\vee-w^\vee=0$, and
$-6\gamma+2u^\vee=0$.  Consequently,
$\operatorname{coker}(T_2-I|V^+)
=\mathbb Z[\gamma]\oplus\mathbb Z[\Delta]
\oplus(\mathbb Z/2)[u^\vee-3\gamma]$.
In this local cokernel $[\gamma]$ has infinite order, whereas the class
denoted $[\gamma]$ in $C_0$ has order two.  Since $p^*$ is injective,
$p^*\alpha_2$ is represented here by $[u^\vee-3\gamma]$.  This determines
the restriction to $M_2$; the restriction to the common boundary $M$ is
determined below, after gluing, by naturality of the universal-coefficient
$\operatorname{Ext}$-term.

Combining the presentations of the two affine fundamental groups in the
homotopy pushout, the $A_1$-coinvariant relations reduce to
$u=2w$ and $2d=3w$, whereas the $A_2$-coinvariant relations reduce to
$u=w$ and $2(d-u)=0$.  Together they give $u=w=0$ and $2d=0$.
The affine relations then become
$3r_1=\widehat\gamma$ and $4r_2=-\widehat\gamma$.  Hence
\begin{equation}\label{eq:H1-B-presentation}
 \begin{split}
 H_1(Y_{\mathrm{fin}};\mathbb Z)
 &=\frac{\mathbb Z\langle\widehat\gamma,d,r_1,r_2\rangle}
 {\mathbb Z\langle2d,\,3r_1-\widehat\gamma,\,
 4r_2+\widehat\gamma\rangle}\\
 &\cong\mathbb Z\langle\mathfrak s\rangle
   \oplus(\mathbb Z/2)\langle d\rangle,
 \end{split}
\end{equation}
where $\widehat\gamma=12\mathfrak s$, $r_1=4\mathfrak s$, and
$r_2=-3\mathfrak s$.  The cohomology
Mayer--Vietoris segment is
\[
 \begin{split}
 H^1(S_1)\oplus H^1(S_2)&\longrightarrow H^1(F)
 \longrightarrow H^2(Y_{\mathrm{fin}})\\
 &\longrightarrow H^2(S_1)\oplus H^2(S_2)
 \longrightarrow H^2(F).
 \end{split}
\]
The cokernel of its first arrow is free cyclic.  Indeed,
$3\gamma$ and $4\gamma$ generate $\gamma$, while
imposing the relations $\eta_1-3\eta_2=u^\vee-w^\vee$ and $\eta_2=0$ gives
$u^\vee=w^\vee$, $\Delta=-2w^\vee$.  Thus this cokernel is generated primitively by
$[w^\vee]$.  By \eqref{eq:finite-fibre-lattice-arithmetic}, the kernel of
the last arrow on the free summands is $\mathbb Z(4Q)$.  In addition, the
class $\alpha_2$ of order two from $S_2$ lies in the full kernel and
restricts trivially
to the torsion-free group $H^2(F)$.  The universal coefficient theorem and
\eqref{eq:H1-B-presentation} give an injection
$\operatorname{Ext}(H_1(Y_{\mathrm{fin}}),\mathbb Z)\cong\mathbb Z/2$
into $H^2(Y_{\mathrm{fin}};\mathbb Z)$; naturality for
$S_2\hookrightarrow Y_{\mathrm{fin}}$ is expressed by
\[
 \begin{array}{ccc}
 \operatorname{Ext}(H_1(Y_{\mathrm{fin}}),\mathbb Z)
   &\lhook\joinrel\longrightarrow&H^2(Y_{\mathrm{fin}};\mathbb Z)\\
 \Big\downarrow&&\Big\downarrow\\
 \operatorname{Ext}(H_1(S_2),\mathbb Z)
   &\lhook\joinrel\longrightarrow&H^2(S_2;\mathbb Z).
 \end{array}
\]
The left arrow is an isomorphism because
$[d-u]\mapsto[d]$ is an isomorphism on torsion.  Hence the nonzero torsion
class restricts to $\alpha_2$.  In particular,
$H^2(Y_{\mathrm{fin}};\mathbb Z)$ contains an element of order two mapping to
$\alpha_2$.  Hence the extension of the $\mathbb Z/2$-summand by the
connecting $\mathbb Z$-summand splits, and
\begin{equation}\label{eq:H2-B}
 H^2(Y_{\mathrm{fin}};\mathbb Z)=
 \mathbb Z b\oplus\mathbb Z Q_{\mathrm{fin}}
 \oplus(\mathbb Z/2)\alpha.
\end{equation}
Here $b$ is the connecting class represented by $w^\vee$ in
$V^+/(\pi_1^*H^1(S_1)+\pi_2^*H^1(S_2))\cong\mathbb Z[w^\vee]$, the class
$Q_{\mathrm{fin}}$ restricts to $4Q$ on $F$, and $\alpha$ is the
unique nonzero torsion class.  Its existence and uniqueness also follow directly
from \eqref{eq:H1-B-presentation} and universal coefficients.

The summand of order two in \eqref{eq:H2-B} is the
$\operatorname{Ext}$-term in the universal coefficient theorem arising
from $\operatorname{Tor}H_1(Y_{\mathrm{fin}};\mathbb Z)$.
The next lemma concerns the homology group $H_2(Y_{\mathrm{fin}};\mathbb Z)$.

\begin{lemma}[Pushforward under the fourfold cover]
\label{lem:fourfold-pushforward}
For the fourfold covering $\pi_2:F\to S_2$, one has
$H_2(S_2;\mathbb Z)=\mathbb Z^2\oplus
(\mathbb Z/2)\langle z\rangle$ and $\pi_{2*}(u\wedge d)=z$.
Consequently $H_2(Y_{\mathrm{fin}};\mathbb Z)$, and hence
$H^3(Y_{\mathrm{fin}};\mathbb Z)$, is torsion free.
\end{lemma}

\begin{proof}
Retain the notation and product model of
\eqref{eq:order-four-product-model}.
Write $\bar k_1$ for the class of $k_1$ in the coinvariants
$\mathbb Z^2/(R-I)\mathbb Z^2$.  The Wang and K\"unneth sequences
therefore give
$H_2(S_2;\mathbb Z)=\mathbb Z^2\oplus
(\mathbb Z/2)\langle z\rangle$, where
$z=[\bar k_1]\times[d]$.
For comparison, the abelianization of $\pi_1(S_1)$ has
$u=2w$, $2d=3w$, and $\widehat\gamma=3r_1$.  Coprimality of $2$ and $3$ gives
$H_1(S_1;\mathbb Z)\cong\mathbb Z^2$, and Poincar\'e duality gives
torsion-free $H_2(S_1;\mathbb Z)$.
For the fourfold covering, the integral pushforward is
\begin{equation}\label{eq:order-two-pushforward}
 \pi_{2*}(u\wedge d)
 =\pi_{2*}(k_1\wedge d)=z.
\end{equation}
Here $\pi_{2*}$ is induced by the quotient covering, not the transfer, so
no factor of four occurs.  Indeed, under the product model above it sends the
fibre loop $k_1$ to its coinvariant $\bar k_1$ and leaves the $d$-circle
unchanged.  Choose the product cell structure in which $e_{k_1}$ and
$e_d$ are the corresponding oriented one-cells.  The $C_4$-orbit of
$e_{k_1}\times e_d$ consists of the four distinct cells with first
directions $k_1,k_2,-k_1,-k_2$, and hence has trivial stabilizer.  The
quotient characteristic map is
\[
 \begin{array}{ccc}
 D^1\times D^1&
 \xrightarrow{\ \chi_{k_1}\times\chi_d\ }&F\\
 \big\Vert&&\big\downarrow{\pi_2}\\
 D^1\times D^1&
 \xrightarrow{\ \chi_{\bar k_1}\times\chi_d\ }&S_2 .
 \end{array}
\]
On cell interiors $\pi_2$ has degree $+1$ with these orientations, and
the boundary gives the coinvariant relation.  Thus the cellular chain map
sends $k_1\times d$ to $[\bar k_1]\times[d]$, proving
\eqref{eq:order-two-pushforward}.
On the other hand, $u\wedge d$ pairs to zero with each of
$\gamma\eta_1,\gamma\eta_2,Q$, so it lies in the kernel of the free part of
$\Phi_2=(\pi_{1*},-\pi_{2*}):H_2(F)\to H_2(S_1)\oplus H_2(S_2)$
and maps onto all target torsion.  The transpose of the free map has image
$L_1+L_2$, which is saturated by
\eqref{eq:finite-fibre-lattice-arithmetic}; hence the free cokernel is
free.  Let $T=\operatorname{Tor}H_2(S_2)$, and let
$\Phi_{2,\mathrm{free}}$ be the composite of $\Phi_2$ with the
maximal free quotient of its target.  The restriction of $\Phi_2$ to
$\ker\Phi_{2,\mathrm{free}}$ lands in $T$; denote it by $t$.  The
exact sequence
\[
 0\longrightarrow T/t(\ker\Phi_{2,\mathrm{free}})
 \longrightarrow\operatorname{coker}\Phi_2
 \longrightarrow\operatorname{coker}\Phi_{2,\mathrm{free}}
 \longrightarrow0
\]
shows, using \eqref{eq:order-two-pushforward}, that
$\operatorname{coker}\Phi_2$ is free.  Homological
Mayer--Vietoris for $Y_1\cup_F Y_2$ then makes
$H_2(Y_{\mathrm{fin}};\mathbb Z)$ free.  The universal coefficient theorem gives
\begin{equation}\label{eq:H3-B-torsion-free}
 H^3(Y_{\mathrm{fin}};\mathbb Z)\ \text{is torsion free}.
\end{equation}
This proves the lemma.\qedhere
\end{proof}

\medskip
\begin{lemma}[Circle actions and the slant product]\label{lem:sweeping-slant}
Let a circle act continuously on a space $Z$, with action map
$m:S^1\times Z\to Z$.  For a loop $a:S^1\to Z$, denote by
$\operatorname{sw}(a)$ the homology class
$m_*([S^1]\times[a])$.  If $\omega\in H^2(Z;\mathbb Z)$, define
$\omega/[S^1]=m^*\omega/[S^1]\in H^1(Z;\mathbb Z)$.  Then
$\langle\omega,\operatorname{sw}(a)\rangle
=\langle\omega/[S^1],[a]\rangle$.
\end{lemma}

\begin{proof}
Naturality of the Kronecker pairing and the defining adjunction for slant
product give
\[
 \begin{aligned}
 \bigl\langle\omega,\,
   m_*([S^1]\times[a])\bigr\rangle
 &=\bigl\langle m^*\omega,\,[S^1]\times[a]\bigr\rangle\\
 &=\bigl\langle m^*\omega/[S^1],\,[a]\bigr\rangle .
 \end{aligned}\qedhere
\]
\end{proof}

To compute the restriction to the cusp boundary, put
$\mathbf i=(i_1,i_2,i_3,i_4)
=\bigl(\gamma u^\vee,\gamma\Delta,
u^\vee\Delta-2\gamma w^\vee,Q\bigr)$.
This is a basis of $I_0$, since replacing $u^\vee w^\vee$ by
$Q=u^\vee w^\vee+3\gamma\Delta$ is a unimodular change of basis.
Choose lifts
$\widetilde{\mathbf i}=(\widetilde i_1,\ldots,\widetilde i_4)$ under
\eqref{eq:cusp-H2-splitting}.  Their restrictions to $M$ project to
$i_1,\ldots,i_4$ in the Wang sequence \eqref{eq:cusp-Wang}.  The two
solid-torus bounds in Lemma~\ref{lem:cusp-retraction} show that these
restrictions have zero $[w^\vee]$- and $[\Delta]$-coordinates.  After
adding $\kappa$ when necessary, they also have zero
$\vartheta$-component.  Write
$\bar i_k=\widetilde i_k|_M$, set
$\widetilde Q=\widetilde i_4$.  Thus
$(\bar i_1,\ldots,\bar i_4,[w^\vee],[\Delta])$ is a basis of
$H^2(M;\mathbb Z)/\operatorname{Tor}$.

The class $b$ may be replaced by $\pm b+\epsilon\alpha$, and
$Q_{\mathrm{fin}}$ by $Q_{\mathrm{fin}}+n b+\epsilon'\alpha$, where
$n\in\mathbb Z$ and $\epsilon,\epsilon'\in\{0,1\}$; these changes do not
alter their images in $H^2(F;\mathbb Z)$.  We now verify that they can be
chosen with boundary restrictions
\begin{equation}\label{eq:B-boundary-restrictions}
 b|_M=[w^\vee],\qquad
 \alpha|_M=\vartheta,\qquad
 Q_{\mathrm{fin}}|_M=4\widetilde Q|_M-[\Delta].
\end{equation}
Since $S_1,S_2$, and $F$ are aspherical and the attaching maps are
coverings, the standard tree-of-spaces universal cover of the homotopy
pushout is contractible.  Thus $Y_{\mathrm{fin}}$ is a $K(\pi,1)$.  We
use the standard description of classes in
$H^2(Y_{\mathrm{fin}};\mathbb Z)$ by central $\mathbb Z$-extensions.  To
find the free $C_0$-coordinates of $b|_M$,
we evaluate the corresponding central commutator on the $A_0$-invariant
fibre loops $w$ and $d$.  Under this description, the connecting class
$b=\delta[w^\vee]$ is obtained by gluing the trivial central extensions of
the two affine groups.  Let $\mathbf z$
denote the central generator, and let $\widetilde\lambda_j$ and
$\widetilde r_j$ be the standard lifts in the $j$-th extension.  Choose
the connecting-map convention
$\widetilde\lambda_1
=\mathbf z^{\,w^\vee(\lambda)}\widetilde\lambda_2$.
The lifts satisfy
$\widetilde r_j\widetilde\lambda_j\widetilde r_j^{-1}
=\widetilde{A_j\lambda}_j$.  Set
$\widetilde a_0=\widetilde r_2^{-1}\widetilde r_1^{-1}$; this lifts
$a_0=r_2^{-1}r_1^{-1}$.  Direct substitution gives
\[
 \widetilde a_0\widetilde\lambda_1\widetilde a_0^{-1}
 =\mathbf z^{\,w^\vee(A_1^{-1}\lambda)-w^\vee(A_0\lambda)}
   \widetilde{A_0\lambda}_1.
\]
Indeed, after conjugation by $\widetilde r_1^{-1}$, changing from the first
extension to the second contributes $w^\vee(A_1^{-1}\lambda)$;
conjugation by $\widetilde r_2^{-1}$ changes the translation to
$A_0\lambda$, and changing back contributes
$-w^\vee(A_0\lambda)$.  For an $A_0$-invariant $\lambda$, the central
commutator is therefore
$w^\vee(A_1^{-1}\lambda)-w^\vee(\lambda)$; reversing the convention changes
only its overall sign.  Since $A_1^{-1}w=-u+2d$ and $A_1^{-1}d=d$, its
image in the free part of $C_0$ is $\pm[w^\vee]$, with no
$[\Delta]$-component.  Changing the sign of $b$ and adding $\alpha$
gives the first formula in \eqref{eq:B-boundary-restrictions}.  The torsion
generator $d$ maps
isomorphically from $H_1(M)$ to the torsion summand of
$H_1(Y_{\mathrm{fin}})$, so
naturality of the universal-coefficient $\operatorname{Ext}$-term gives
the second formula.

Because $A_1d=A_2d=d$, the affine twists commute with translation by the
$d$-circle, which therefore acts globally
on $Y_{\mathrm{fin}}$.  We have oriented
$\operatorname{sw}_d(a_0)$ by $a_0\wedge d$.  In applying
Lemma~\ref{lem:sweeping-slant}, orient the acting circle by $-d$, so that
$(-d)\wedge a_0=a_0\wedge d$, and write slant product as
${-}/[-d]$.  Contraction on a smooth fibre then gives
$(4Q)/[-d]=12\gamma$.  Choose the generator
$\Theta\in H^1(Y_{\mathrm{fin}};\mathbb Z)$ normalized by
$\Theta(\mathfrak s)=1$.  Then $\Theta|_F=12\gamma$; injectivity of
$H^1(Y_{\mathrm{fin}})\to H^1(F)$ therefore gives
$Q_{\mathrm{fin}}/[-d]=\Theta$.  With the chosen orientation, the suspension torus
is the image of the product of the acting circle with the loop $a_0$, so
Lemma~\ref{lem:sweeping-slant} gives
$\langle Q_{\mathrm{fin}}|_M,\operatorname{sw}_{d}(a_0)\rangle
=\langle Q_{\mathrm{fin}}/[-d],a_0\rangle=\Theta(a_0)$.
Now $a_0=r_2^{-1}r_1^{-1}=-\mathfrak s$ in
$H_1(Y_{\mathrm{fin}})$, so $\Theta(a_0)=-1$.  The corresponding
suspension torus bounds in $A$, and hence pairs trivially with
$4\widetilde Q|_M$.  Consequently,
$\langle Q_{\mathrm{fin}}|_M-4\widetilde Q|_M,
\operatorname{sw}_{d}(a_0)\rangle=-1$.
This is the coefficient of $[\Delta]$ in
$Q_{\mathrm{fin}}|_M-4\widetilde Q|_M$, with the stated sign.

For the Mayer--Vietoris difference map we use the convention
$\operatorname{res}_A-\operatorname{res}_{Y_{\mathrm{fin}}}$.
In the domain basis
$(\widetilde{\mathbf i},\kappa,b,Q_{\mathrm{fin}})$ and the target basis
$(\bar i_1,\bar i_2,\bar i_3,\bar i_4,[w^\vee],[\Delta])$ of the free
parts, its matrix is
\begin{equation}\label{eq:boundary-free-matrix}
 \begin{pmatrix}
 1&0&0&0&0& 0& 0\\
 0&1&0&0&0& 0& 0\\
 0&0&1&0&0& 0& 0\\
 0&0&0&1&0& 0&-4\\
 0&0&0&0&0&-1& 0\\
 0&0&0&0&0& 0& 1
 \end{pmatrix}.
\end{equation}
Its six nonzero Smith invariants are all one.  The remaining torsion row is
the map
\begin{equation}\label{eq:boundary-torsion-row}
 \mathbb Z\langle\kappa\rangle\oplus
 (\mathbb Z/2)\langle\alpha\rangle
 \longrightarrow(\mathbb Z/2)\langle\vartheta\rangle,
 \qquad (n,\epsilon)\longmapsto n-\epsilon\pmod2.
\end{equation}
Equations \eqref{eq:boundary-free-matrix} and
\eqref{eq:boundary-torsion-row} are the free and torsion parts of the same
integral map.

\subsection{The Mayer--Vietoris map in cohomological degree two}

The cusp neighbourhood and the neighbourhood of the multiplicity-four fibre
contribute classes whose restrictions to the common boundary have order two.
Equations~\eqref{eq:kappa-boundary}
and \eqref{eq:B-boundary-restrictions} show that their restrictions are equal.

\begin{proof}[Proof of Proposition~\ref{prop:degree-two-mv}]
The classes $\kappa$ and $\alpha$ have the same nonzero restriction,
$\kappa|_M=\vartheta=\alpha|_M$,
by \eqref{eq:kappa-boundary} and
\eqref{eq:B-boundary-restrictions}.  To determine the full kernel, use
\eqref{eq:cusp-H2-splitting}, \eqref{eq:cusp-Wang}, and
\eqref{eq:B-boundary-restrictions}.  The $I_0$-part first cancels the
$4Q$-part of a multiple of $Q_{\mathrm{fin}}$.  Its remaining coordinate in
$C_0$ is the same multiple of the primitive class $-[\Delta]$, so
that multiple is zero.  The primitive $[w^\vee]$-coordinate then excludes the
coefficient of the connecting generator $b$.  The remaining pairs are
$(n\kappa,\epsilon\alpha)$, where $n\in\mathbb Z$,
$\epsilon\in\{0,1\}$, and $n\equiv\epsilon\pmod2$.
They form the infinite cyclic group generated by
\eqref{eq:x-generator}.  Since $2\alpha=0$,
\eqref{eq:cusp-component-class} becomes
$[D_0]=(2\kappa,0)=2x$ and $[D_1]=-2x$.
\end{proof}

\subsection{Van Kampen and Mayer--Vietoris}

\begin{proof}[Proof of Proposition~\ref{prop:integral-topology}]
Let $r_1,r_2$ be the affine generators.  Lemma~\ref{lem:cusp-retraction}
and van Kampen give
\[
 \begin{split}
 \pi_1(Y)=\langle\Lambda^+,r_1,r_2\mid{}&
 r_j\lambda r_j^{-1}=A_j\lambda,
 \quad r_1^3=v_1,\quad r_2^4=v_2,\\
 &w=d=1,\quad r_1r_2=1\rangle.
 \end{split}
\]
The normal closure of $w,d$ also kills $u$, for
$A_2w=-u+2d$.  If $c$ is the image of $\widehat\gamma$, and
$x_1=r_1,x_2=r_2$, the remaining relations are
$x_1^3=c$, $x_2^4=c^{-1}$, and $x_1x_2=1$.
Thus $c=x_1^3=x_1^4$, and every generator is trivial.  Hence
$\pi_1(Y)=1$.

The map on first homology retains the integral torsion.  From
\eqref{eq:cusp-monodromy-homology},
$H_1(M)=\mathbb Z\langle\widehat\gamma,u,a_0\rangle
\oplus(\mathbb Z/2)\langle d\rangle$.
The map to the two pieces is
\begin{equation*}
 \begin{array}{c|cccc}
  &\widehat\gamma&u&a_0&d\\ \hline
 H_1(A)&\widehat\gamma&u&0&0\\
 H_1(Y_{\mathrm{fin}})&12\mathfrak s&0&-\mathfrak s&d.
 \end{array}
\end{equation*}
Its determinant on the free summands is $-1$, and it is the identity on
the summand of order two.  It is therefore an isomorphism.

Cohomological Mayer--Vietoris consequently identifies $H^2(Y;\mathbb Z)$
with the kernel in
Proposition~\ref{prop:degree-two-mv}.  Hence
$H^2(Y;\mathbb Z)=\mathbb Zx$ and
$[D_0]=2x$, $[D_1]=-2x$.

The same boundary formulas show that
$H^2(A)\oplus H^2(Y_{\mathrm{fin}})\to H^2(M)$ is surjective.  Its
composition with $H^2(M)\to I_0$ is surjective, and its image also contains
$\vartheta$ and $[w^\vee]$.  The difference of
$4\widetilde Q$ and $Q_{\mathrm{fin}}$ gives the primitive class
$[\Delta]$.  Mayer--Vietoris therefore injects
\begin{equation}\label{eq:H3-Y-injection}
 H^3(Y;\mathbb Z)\hookrightarrow
 H^3(A;\mathbb Z)\oplus H^3(Y_{\mathrm{fin}};\mathbb Z).
\end{equation}
The target is torsion free by \eqref{eq:cusp-homology} and
\eqref{eq:H3-B-torsion-free}.

For the Euler characteristic, use $e_c$, which is
additive for this locally closed stratification of the base and
multiplicative for the locally trivial part of the fibration.  Since $Y$
is compact, $e(Y)=e_c(Y)$.  Each reduced multiple fibre is a free finite
quotient of a real four-torus, so its Euler characteristic is zero; the
scheme-theoretic multiplicity does not change its underlying topological
space.  Over the remaining punctured-base stratum the fibre is a real
four-torus, and hence
$e_c(f^{-1}(B^\circ))=e_c(B^\circ)e(T^4)=0$.
The cusp fibre is the only contribution, and
\eqref{eq:cusp-homology} gives $e(W)=e(A)=4$.  Therefore $e(Y)=4$.
Simple connectivity, $H^2(Y)=\mathbb Z$, Poincar\'e duality, and this
Euler number give $b_3(Y)=0$.  The injection
\eqref{eq:H3-Y-injection} then gives $H^3(Y;\mathbb Z)=0$, including
torsion.  The injection
$\operatorname{Ext}(H_2(Y;\mathbb Z),\mathbb Z)\hookrightarrow
H^3(Y;\mathbb Z)=0$ shows that $H_2(Y;\mathbb Z)$ is torsion free,
and integral Poincar\'e duality supplies all remaining groups.
\end{proof}

\section{Characteristic classes}
\label{sec:characteristic-classes}

Let $D$ be either component of the cusp fibre, and choose the sign of
$x\in H^2(Y;\mathbb Z)$ so that $[D]=2x$.  The normalization of $D$
and its conductor determine the cubic form and the characteristic classes.

\begin{proof}[Proof of Proposition~\ref{prop:characteristic-classes}]
Let $\nu:S\to D$, $C_1,\ldots,C_6$, and
$C_{\mathrm{dbl}}$ be as in Lemma~\ref{lem:cusp-double-curve}.  That lemma
gives $D^3=0$, $D\cdot C_{\mathrm{dbl}}=-2$, and
$\deg N_{C_{\mathrm{dbl}}/Y}=-2$.
Thus $8x^3=0$; the group $H^6(Y;\mathbb Z)\cong\mathbb Z$ is
torsion free, so $x^3=0$.  Moreover,
$[D]=2x$ gives $x(C_{\mathrm{dbl}})=-1$, and hence
$C_{\mathrm{dbl}}$ is a primitive generator of
$H_2(Y;\mathbb Z)$.  Adjunction on this smooth rational curve gives
$\langle c_1(Y),C_{\mathrm{dbl}}\rangle
=\deg T\mathbb P^1+\deg N_{C_{\mathrm{dbl}}/Y}=2-2=0$.
It follows that
\begin{equation}\label{eq:c1-zero}
 c_1(Y)=0,\qquad w_2(Y)=0.
\end{equation}
Here $c_1(Y)=0$ is an equality of integral topological Chern classes; it
does not assert a holomorphic trivialization of $K_Y$.

Let $\phi:C_1\to C_4$ be the isomorphism that identifies the two branches.
The normalization sequence for $D$ is
\begin{equation}\label{eq:normalization-sequence}
 0\longrightarrow\mathcal O_D
 \longrightarrow\nu_*\mathcal O_S
 \xrightarrow{\operatorname{res}_{C_1}
  -\phi^*\operatorname{res}_{C_4}}
 \mathcal O_{C_{\mathrm{dbl}}}\longrightarrow0.
\end{equation}
At a triple point the local algebra of $D$ and its normalization is
\[
 \frac{\mathbb C\{z_1,z_2,z_3\}}{(z_1z_2)}
 \longrightarrow
 \mathbb C\{z_1,z_3\}\oplus\mathbb C\{z_2,z_3\}.
\]
The cokernel is $\mathbb C\{z_3\}$, via the difference of the two
restrictions at $z_1=z_2=0$.  Hence the difference map in
\eqref{eq:normalization-sequence} remains surjective at the triple point and
has no additional zero-dimensional cokernel.  Since $S$ is
rational and $C_{\mathrm{dbl}}\cong\mathbb P^1$, equation
\eqref{eq:normalization-sequence} gives $\chi(\mathcal O_D)=1-1=0$.
On the smooth threefold $Y$, apply the Hirzebruch--Riemann--Roch theorem
to the $K$-theory identity supplied by
$0\to\mathcal O_Y(-D)\to\mathcal O_Y\to\mathcal O_D\to0$.
Thus $[\mathcal O_D]=1-[\mathcal O_Y(-D)]$ and
$\operatorname{ch}(\mathcal O_D)=1-e^{-D}$.  Using
\eqref{eq:c1-zero}, the degree-six part of Hirzebruch--Riemann--Roch is
$\chi(\mathcal O_D)
=\int_Y(D^3/6+D\,c_2(Y)/12)$.
Since $D^3=\chi(\mathcal O_D)=0$, the Riemann--Roch formula gives
$D\,c_2(Y)=0$.  Choose $y\in H^4(Y;\mathbb Z)$ so that
$\langle x\smile y,[Y]\rangle=1$, and write
$c_2(Y)=ay$.  Since $D=2x$, we obtain $2a=0$, and hence
$c_2(Y)=0$ and $p_1(Y)=c_1(Y)^2-2c_2(Y)=0$.
\end{proof}

%% file: sections/period-torsor-appendix.tex
\section{A proof of the period-function lemma}
\label{app:period-torsor}

The existence of the functions $\tau,\mu,\beta$ with the required
transformation laws is proved in \citep[Theorem~3.4]{Alpoge2026}.
Lemma~\ref{lem:an-period-functions} records the form needed here.  We give a
condensed, self-contained proof in the notation of
Section~\ref{sec:analytic-construction}, following the four stages of the
argument in \citep[Sections~3.1--3.4]{Alpoge2026}.

\begin{proof}[Proof of Lemma~\ref{lem:an-period-functions}]
\smallskip
\noindent\emph{The modular lift and its cusp normalization.}
We use the standard facts about the modular group and the modular invariant
collected in
\citep[Chapter~VII, Sections~1.1--1.2 and~3.1--3.3]{Serre1973}.  Let $j$ have its usual
normalization $j(i)=1728$, and set $\mathcal J=1728t\circ\pi$ and
$Z^\circ=\mathfrak H_z\setminus\mathcal J^{-1}\{0,1728\}$.
After the elliptic orbits are removed, $j$ is an unramified modular
covering of $\CC\setminus\{0,1728\}$.  The fundamental group of
$Z^\circ$ is normally generated by small circles around the deleted
points.  Their $\mathcal J$-images are cubes of meridians about zero and fourth
powers of meridians about $1728$.  Since the corresponding modular deck
transformations have orders three and two, respectively, the
covering-space lifting criterion gives a lift
$\tau:Z^\circ\to\mathfrak H$ of $\mathcal J$.  In a bounded-disc model of
$\mathfrak H$, removability and the maximum principle extend the lift
across all deleted points.

For each $g\in\Gamma_{\mathrm{orb}}$, the maps $\tau\circ g$ and $\tau$ lift the
same function.  Thus $\tau(gz)=\phi(g)\tau(z)$ for a homomorphism
$\phi:\Gamma_{\mathrm{orb}}\to\operatorname{PSL}_2(\ZZ)$.  Normalize
$\tau(z_1)=\rho=e^{\pi i/3}$.  The map $\mathcal J$ has order three at
$z_1$, while $j$ has order three at $\rho$, so $\tau-\rho$ has
order one.  Comparing its rotation with
$s_1(g_1z)=e^{-2\pi i/3}s_1(z)$ forces
$\phi(g_1)=\sigma_1$, where
$\sigma_1(\tau)=(\tau-1)/\tau$.
At $z_2$, the function $\mathcal J-1728$ has order four whereas
$j-1728$ has order two at $i$.  Hence $\tau-\tau(z_2)$ has order
two.  In particular $\phi(g_2)\ne1$, for otherwise an invariant germ
would have order divisible by four.  Since the modular group has no
element of order four, $\phi(g_2)$ has order two.

To determine $P=\phi(g_0)$, take $R$ large.  Every component
of $\{|j|>R\}$ is contained in a horoball based at a rational cusp, and
distinct such horoballs are disjoint.  Indeed, this follows from the
standard modular fundamental domain together with
$\Im((a\tau+b)/(c\tau+d))=\Im\tau/|c\tau+d|^2$.
The chosen component $\widetilde U_0$ maps into one such horoball, so $P$
fixes the corresponding rational cusp.  It is not the identity.  Otherwise
$q=e^{2\pi i\tau}$ would descend to the punctured $t_c$-disc and
extend boundedly across zero.  The facts that $j\to\infty$ and
$|q(0)|\ne1$ would give $q(0)=0$, but the single-valued logarithm
$2\pi i\tau$ would force the winding number of $q$ to vanish.  Thus
$P$ is nontrivial parabolic.

Choose its trace-two lift in the form
\[
 P=I+k
 \begin{pmatrix}
 -pr&p^2\\-r^2&pr
 \end{pmatrix},
 \qquad \gcd(p,r)=1,\quad k\ne0.
\]
The element $\phi(g_2)=\sigma_1^{-1}P^{-1}$ has trace zero.  Substitution
gives $-k(p^2-pr+r^2)=1$.
Consequently $k=-1$ and
$(p,r)=\pm(1,0),\pm(0,1),\pm(1,1)$.  Conjugating by a power of $\sigma_1$,
which preserves the normalization at $z_1$, gives
$P(\tau)=\tau-1$ and $\phi(g_2)(\tau)=-1/\tau$.
Hence \eqref{eq:an-tau-laws} holds with the stated cusp normalization.

The Fourier expansion of $j$, together with
$j(\tau)=1728/t_c$, gives $q=e^{2\pi i\tau}=t_cu_1(t_c)$
for a holomorphic unit $u_1$.  After shrinking the cusp disc, choose a
holomorphic logarithm and set
$h(t_c)=(2\pi i)^{-1}\log u_1(t_c)$ and $s=\tau-h(t_c)$.
Then \eqref{eq:an-log-coordinate} holds.  After shrinking $U_0^\times$, the
maps from $\widetilde U_0$ and from a half-plane in the $s$-plane to
$U_0^\times$ are universal covers with the same deck generator; hence
$s$ is a logarithmic coordinate.  In particular,
$\Im\tau\to\infty$ at the cusp.

\smallskip
\noindent\emph{Construction of $\mu$.}
For $\mu$, differences of solutions
of \eqref{eq:an-mu-laws} obey $\nu(g_1z)=-\nu/\tau$ and
$\nu(g_2z)=\nu/\tau$.
The products of these factors around $g_1^3$ and $g_2^4$ are one, so
they extend to a factor of automorphy for $\Gamma_{\mathrm{orb}}$.  Let
$\mathcal L$ be the sheaf on $B$ whose local sections are holomorphic
functions satisfying the resulting equivariance law.  At the cusp, its
sections are required to be holomorphic in $t_c$.  We claim
$\mathcal L\simeq\cO_{\PP^1}(-p_0)$.
To see that $\mathcal L$ is locally free at $p_1,p_2$ and identify it,
one must account for the stabilizers.  On the simply connected cover the
meromorphic function
$F_{\mathcal L}=E_4(\tau)^2E_6(\tau)^{1/2}/
\Delta_{\mathrm{mod}}(\tau)$
is defined because the pullback $E_6(\tau(z))$ has zeros of even order
two over $p_2$.  The modular transformation laws give precisely the
preceding homogeneous factors, up to a possible sign character.  More
precisely, $F_{\mathcal L}(g_1z)=-F_{\mathcal L}(z)/\tau(z)$ and
$F_{\mathcal L}(g_2z)=F_{\mathcal L}(z)/\tau(z)$.
The relation $g_1^3=1$ fixes the sign in the first formula, while comparison
of the leading term in $s_2(g_2z)=-is_2(z)$ fixes the sign in the second.

Upstairs, $F_{\mathcal L}$ has orders two and one at $z_1,z_2$, and
at the cusp it is $t_c^{-1}$ times a unit.  If $\nu$ is a local
holomorphic section, the quotient $\nu/F_{\mathcal L}$ is invariant.  At
$z_1$ its order is at least $-2$ and divisible by three, while at
$z_2$ its order is at least $-1$ and divisible by four.  Thus
$\nu/F_{\mathcal L}$ descends to a function holomorphic at $p_1$ and
$p_2$.  At $p_0$, the
regularity condition is exactly
$\mathcal L_{p_0}=\mathfrak m_{p_0}F_{\mathcal L}$.  Thus
$\mathcal L\simeq\cO_{\PP^1}(-p_0)$.

The affine laws \eqref{eq:an-mu-laws} close after three and four
iterations:
\[
\begin{aligned}
(\tau,\mu)&\xmapsto{g_1}
 \left(\frac{\tau-1}{\tau},\frac{1-\mu}{\tau}\right)
 \xmapsto{g_1}
 \left(\frac{-1}{\tau-1},\frac{\tau-1+\mu}{\tau-1}\right)
 \xmapsto{g_1}(\tau,\mu),\\
(\tau,\mu)&\xmapsto{g_2}
 \left(\frac{-1}{\tau},1+\frac\mu\tau\right)
 \xmapsto{g_2}(\tau,1-\tau-\mu)
 \xmapsto{g_2}
 \left(\frac{-1}{\tau},\frac{1-\mu}{\tau}\right)
 \xmapsto{g_2}(\tau,\mu).
\end{aligned}
\]
Explicit local solutions are $\mu=(2-\tau)/3$ at $p_1$,
$\mu=(1-\tau)/2$ at $p_2$, and $\mu=0$ on $\widetilde U_0$.
At an ordinary point, choose a solution on one sheet and extend it
equivariantly.  The solution sheaf is a torsor under $\mathcal L$, and
$H^0(\PP^1,\cO(-1))=H^1(\PP^1,\cO(-1))=0$.
Thus the local solutions glue uniquely to $\mu$.

\smallskip
\noindent\emph{Construction of $\beta$.}
For $\beta$, put $\varphi_1=2-6(1-\mu)^2/\tau$ and
$\varphi_2=-3-6\mu^2/\tau$.
Substitution in \eqref{eq:an-tau-laws}--\eqref{eq:an-mu-laws} gives
\begin{equation}\label{eq:an-beta-zero-sums}
 \sum_{k=0}^{2}\varphi_1(g_1^kz)=0,
 \qquad
 \sum_{k=0}^{3}\varphi_2(g_2^kz)=0.
\end{equation}
Hence the affine laws \eqref{eq:an-beta-laws} close.  A local solution at
$p_j$ is $m_j^{-1}\sum_{k=0}^{m_j-1}k\varphi_j(g_j^kz)$,
because \eqref{eq:an-beta-zero-sums} makes its increment telescope to
$\varphi_j$.  At the cusp take $\beta=-\tau$, so that
$b_c=\beta+\tau$ is regular.  Differences of solutions form
$\cO_{\PP^1}$, and $H^1(\PP^1,\cO)=0$, so a global $\beta$ exists,
unique up to a constant.

\smallskip
\noindent\emph{Nondegeneracy.}
Finally, direct substitution shows that the function $\mathcal D$ in
\eqref{eq:an-D-negative} is
$\Gamma_{\mathrm{orb}}$-invariant.  It therefore descends to a continuous
function on $B\setminus\{p_0\}$, and at the cusp
$\mathcal D\leq \Im b_c-\Im\tau\to-\infty$.
It is bounded above off a small cusp disc.  Adding a sufficiently negative
imaginary constant to $\beta$ therefore gives
\eqref{eq:an-D-negative} everywhere.
\end{proof}